\documentclass[11pt,draft]{amsart}

\usepackage{a4,cite,amsmath,amssymb,amsthm,gastex}

\newtheorem{theorem}{Theorem}[section]

\newtheorem{proposition}[theorem]{Proposition}

\newtheorem{lemma}[theorem]{Lemma}

\newtheorem{corollary}[theorem]{Corollary}

\newtheorem{question}[theorem]{Question}

\DeclareMathOperator{\CR}{CR}

\DeclareMathOperator{\Eq}{Eq}

\DeclareMathOperator{\Gr}{Gr}

\DeclareMathOperator{\Nil}{Nil} 

\DeclareMathOperator{\var}{var}

\makeatletter

\@addtoreset{equation}{section}

\renewcommand*\subjclass[2][2010]{\def\@subjclass{#2}\@ifundefined{subjclassname@#1}{\ClassWarning{\@classname}{Unknown edition (#1) of Mathematics Subject Classification; using '2010'.}}{\@xp\let\@xp\subjclassname\csname subjclassname@#1\endcsname}}

\makeatother

\renewcommand{\subjclassname}{\textup{2010} Mathematics Subject Classification}

\begin{document}

\title[Special elements of the lattice of epigroup varieties]{Special elements of the lattice\\
of epigroup varieties}

\thanks{Supported by the Ministry of Education and Science of the Russian Federation (project 2248), by a grant of the President of the Russian Federation for supporting of leading scientific schools of the Russian Federation (project 5161.2014.1) and by Russian Foundation for Basic Research (grant 14-01-00524).}

\author{V.\,Yu.\,Shaprynski\v{\i}}

\address{Ural Federal University, Institute of Mathematics and Computer Science, Lenina 51, 620000 Ekaterinburg, Russia}

\email{vshapr@yandex.ru,\,dmitry.skokov@gmail.com,\,bvernikov@gmail.com}

\author{D.\,V.\,Skokov}

\author{B.\,M.\,Vernikov}

\begin{abstract}
We study special elements of eight types (namely, neutral, standard, costandard, distributive, codistributive, modular, lower-modular and upper-modular elements) in the lattice \textbf{EPI} of all epigroup varieties. Neutral, standard, costandard, distributive and lower-modular elements are completely determined. A strong necessary condition and a sufficient condition for modular elements are found. Modular elements are completely classified within the class of commutative varieties, while codistributive and upper-modular elements are completely determined within the wider class of strongly permutative varieties. It is verified that an element of \textbf{EPI} is costandard if and only if it is neutral; is standard if and only if it is distributive; is modular whenever it is lower-modular; is neutral if and only if it is lower-modular and upper-modular simultaneously. We found also an application of results concerning neutral and lower-modular elements of \textbf{EPI} for studying of definable sets of epigroup varieties.
\end{abstract}



\maketitle

\section{Introduction and summary}
\label{introduction}

\subsection{Semigroup pre-history}
\label{introduction semigroups}

The main object we examine in this article is the lattice of all epigroup varieties. But our considerations are motivated by some earlier investigations of the lattice of semigroup varieties and closely related with these investigations. To make clearer a context and motivations of our considerations, we start with a brief explanation of the `semigroup pre-history' of the present work.

One of the main branches of the theory of semigroup varieties is an examination of lattices of semigroup varieties (see the survey~\cite{Shevrin-Vernikov-Volkov-09}). If $\mathcal V$ is a variety then $L(\mathcal V)$ stands for the subvariety lattice of $\mathcal V$ under the natural order (the class-theor\-etical inclusion). The lattice operations in $L(\mathcal V)$ are the (class-theor\-etical) intersection denoted by $\mathcal{X\wedge Y}$ and the join $\mathcal{X\vee Y}$, i.\,e., the least subvariety of $\mathcal V$ containing both $\mathcal X$ and $\mathcal Y$.

There are a number of articles devoted to an examination of identities (first of all, the distributive, modular or Arguesian laws) and some related restrictions (such as semimodularity or semidistributivity) in lattices of semigroup varieties, and many considerable results are obtained here. In particular, semigroup varieties with modular, Arguesian or semimodular subvariety lattice were completely classified and deep results concerning semigroup varieties with distributive subvariety lattices (related to a description of such varieties modulo group ones) were obtained. An overview of all these results may be found in~\cite[Section~11]{Shevrin-Vernikov-Volkov-09}.

The results mentioned above specify, so to say, `globally' modular or distributive parts of the lattice of semigroup varieties. The following natural step is to examine varieties that guarantee modularity or distributivity, so to say, in their `neighborhood'. Saying so, we take in mind special elements in the lattice of semigroup varieties. There are many types of special elements that are considered in lattice theory. Recall definitions of some of them. An element $x$ of a lattice $\langle L;\vee,\wedge\rangle$ is called

\smallskip

\emph{neutral} if $(x\vee y)\wedge(y\vee z)\wedge(z\vee x)=(x\wedge y)\vee(y\wedge z)\vee(z\wedge x)$ for all $y,z\in L$;

\smallskip

\emph{standard} if $(x\vee y)\wedge z=(x\wedge z)\vee(y\wedge z)$ for all $y,z\in L$;

\smallskip

\emph{distributive} if $x\vee(y\wedge z)=(x\vee y)\wedge(x\vee z)$ for all $y,z\in L$;

\smallskip

\emph{modular} if $(x\vee y)\wedge z=(x\wedge z)\vee y$ for all $y,z\in L$ with $y\le z$;

\smallskip

\emph{upper-modular} if $(z\vee y)\wedge x=(z\wedge x)\vee y$ for all $y,z\in L$ with $y\le x$.

\smallskip

\noindent\emph{Costandard}, \emph{codistributive} and \emph{lower-modular} elements are defined dually to standard, distributive and upper-modular ones. There is a number of interrelations between types of elements we consider. It is evident that a neutral element is both standard and costandard; a standard or costandard element is modular; a [co]distributive element is lower-modular [upper-modular]. It is well known also that a [co]standard element is [co]distributive (see~\cite[Theorem~253]{Gratzer-11}, for instance). So, eight types of elements defined above form a partially ordered set under class-theor\-etical inclusion pictured on Fig.~\ref{classes}.

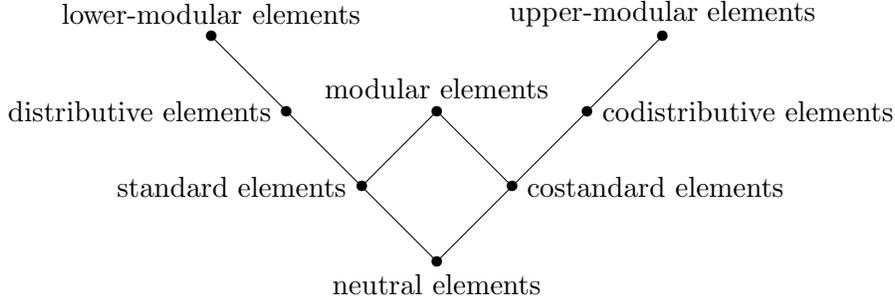
\begin{figure}[tbh]
\begin{center}
\unitlength=1mm
\special{em:linewidth .4pt}
\linethickness{.4pt}
\begin{picture}(118,40)
\gasset{AHnb=0}
\drawline(27,35)(57,5)(87,35)
\drawline(47,15)(57,25)(67,15)
\put(57,5){\circle*{1.33}}
\put(47,15){\circle*{1.33}}
\put(67,15){\circle*{1.33}}
\put(37,25){\circle*{1.33}}
\put(57,25){\circle*{1.33}}
\put(77,25){\circle*{1.33}}
\put(27,35){\circle*{1.33}}
\put(87,35){\circle*{1.33}}
\put(57,2){\makebox(0,0)[cc]{neutral elements}}
\put(45,15){\makebox(0,0)[rc]{standard elements}}
\put(69,15){\makebox(0,0)[lc]{costandard elements}}
\put(35,25){\makebox(0,0)[rc]{distributive elements}}
\put(79,25){\makebox(0,0)[lc]{codistributive elements}}
\put(57,28){\makebox(0,0)[cc]{modular elements}}
\put(27,38){\makebox(0,0)[cc]{lower-modular elements}}
\put(87,38){\makebox(0,0)[cc]{upper-modular elements}}
\end{picture}
\caption{Special elements in abstract lattices}
\label{classes}
\end{center}
\end{figure}

Note that special elements play an essential role in the abstract lattice theory (see~\cite[Section~III.2]{Gratzer-11} or~\cite[Sections~2.1 and~2.2]{Stern-99}, for instance). For instance, if an element $x$ of a lattice $L$ is neutral then $L$ is decomposed into a subdirect product of its intervals $(x]=\{y\in L\mid y\le x\}$ and $[x)=\{y\in L\mid x\le y\}$ (see~\cite[Theorem~254]{Gratzer-11}). Thus, the knowledge of what elements of a lattice are neutral, gives the important information on a structure of the lattice as a whole.

All types of elements mentioned above are intensively and successfully studied with respect to the lattice \textbf{SEM} of all semigroup varieties. For brevity, a semigroup variety that is a neutral element of the lattice \textbf{SEM} is called a \emph{neutral in} \textbf{SEM} variety. Analogous convention is applied for all other types of special elements. Results about special elements in \textbf{SEM} are overviewed in the recent survey~\cite{Vernikov-15}. In particular,
\begin{itemize}
\item neutral in \textbf{SEM} varieties were completely determined in~\cite{Volkov-05};
\item it is proved that a semigroup variety is costandard in \textbf{SEM} if and only if it is neutral in \textbf{SEM}~\cite{Vernikov-11}; thus, in view of the previous result, costandard in \textbf{SEM} varieties are completely classified;
\item distributive in \textbf{SEM} varieties were completely classified in~\cite{Vernikov-Shaprynskii-10};
\item in fact, standard in \textbf{SEM} varieties are completely described in~\cite{Vernikov-Shaprynskii-10} too because the results of this work readily imply that a semigroup variety is standard in \textbf{SEM} if and only if it is distributive in \textbf{SEM} (see comments after Theorem~3.3 in the survey~\cite{Vernikov-15});
\item a strong necessary conditions for modular in \textbf{SEM} varieties were discovered in~\cite{Jezek-McKenzie-93}\footnote{Note that paper~\cite{Jezek-McKenzie-93} deals with the lattice of equational theories of semigroups, that is, the dual of \textbf{SEM} rather than the lattice \textbf{SEM} itself. When reproducing results from~\cite{Jezek-McKenzie-93}, we adapt them to the terminology of the present article.} and~\cite{Vernikov-07a} (these results are reproved in a simpler way in~\cite{Shaprynskii-12a});
\item a sufficient condition for modular in \textbf{SEM} varieties was found in~\cite{Vernikov-Volkov-88} and rediscovered in~\cite{Jezek-McKenzie-93};
\item commutative modular in \textbf{SEM} varieties were completely determined in~\cite{Vernikov-07a};
\item lower-modular in \textbf{SEM} varieties were completely classified in~\cite{Shaprynskii-Vernikov-10};
\item commutative upper-modular in \textbf{SEM} varieties were completely classified in~\cite{Vernikov-08b}; it is noted in~\cite{Vernikov-08a} that this result may be expanded on wider class of strongly permutative varieties without any change (a definition of strongly permutative varieties see in Subsection~\ref{introduction epigroups} below);
\item strongly permutative codistributive varieties were completely described in~\cite{Vernikov-11}.
\end{itemize}
Note that the articles~\cite{Vernikov-08c,Vernikov-11} contain some other results concerning upper-modular and codistributive elements in \textbf{SEM}.

\subsection{Epigroups}
\label{introduction epigroups}

A considerable attention in the semigroup theory is devoted to semigroups equipped by an additional unary operation. Such algebras are said to be \emph{unary semigroups}. As concrete types of unary semigroups, we mention completely regular semigroups (see~\cite{Petrich-Reilly-99}), inverse semigroups (see~\cite{Petrich-84}), semigroups with involution etc.

One more natural type of unary semigroups is epigroups. A semigroup $S$ is called an \emph{epigroup} if, for any element $x$ of $S$, there is a natural $n$ such that $x^n$ is a \emph{group element} (this means that $x^n$ lies in some subgroup of $S$). Extensive information about epigroups may be found in the fundamental work~\cite{Shevrin-94} by L.\,N.\,Shevrin and the survey~\cite{Shevrin-05} by the same author. The class of epigroups is very wide. In particular, it includes all periodic semigroups (because some power of each element in such a semigroup lies in some its finite cyclic subgroup) and all completely regular semigroups (in which all elements are group ones). The unary operation on an epigroup is defined by the following way. If $S$ is an epigroup and $x\in S$ then some power of $x$ lies in a maximal subgroup of $S$. We denote this subgroup by $G_x$. The unit element of $G_x$ is denoted by $x^\omega$. It is well known (see~\cite{Shevrin-94}, for instance) that the element $x^\omega$ is well defined and $xx^\omega=x^\omega x\in G_x$. We denote the element inverse to $xx^\omega$ in $G_x$ by $\overline x$. The map $x\longmapsto\,\overline x$ is just the mentioned unary operation on an epigroup $S$. The element $\overline x$ is called \emph{pseudoinverse} to $x$. Throughout this paper, we consider epigroups as algebras with the operations of multiplication and pseudoinversion. In particular, this allows us to consider varieties of epigroups as algebras with the two mentioned operations. An idea to examine epigroups in the framework of the theory of varieties was promoted by L.\,N.\,Shevrin in~\cite{Shevrin-94} (see also~\cite{Shevrin-05}). An overview of first results obtained here may be found in~\cite[Section~2]{Shevrin-Vernikov-Volkov-09}.

If $S$ is a \emph{completely regular} semigroup (i.\,e., the union of groups) and $x\in S$ then $\overline x$ is the element inverse to $x$ in the maximal subgroup containing $x$. Thus, the operation of pseudoinversion on a completely regular semigroup coincides with the unary operation traditionally considered on completely regular semigroups. We see that varieties of completely regular semigroups (considered as unary semigroups) are varieties of epigroups in the sense defined above. Further, it is well known and may be easily checked that in every periodic epigroup the operation of pseudoinversion may be expressed in terms of multiplication (see~\cite{Shevrin-94}, for instance). This means that periodic varieties of epigroups may be identified with periodic varieties of semigroups.

It seems to be very natural to examine all restrictions on semigroup varieties mentioned in Subsection~\ref{introduction semigroups} for epigroup varieties. This is realized in~\cite{Vernikov-Skokov-16,Vernikov-Volkov-Shaprynskii-16+} for identities and related restrictions to subvariety lattice. In particular, epigroup varieties with modular, Arguesian or semimodular subvariety lattice are completely classified and epigroup analogs of results concerning semigroup varieties with distributive subvariety lattice are obtained there. In the present article, we start with an examination of special elements in the lattice \textbf{EPI} of all epigroup varieties. We consider here elements of all mentioned above eight types in \textbf{EPI}. For brevity, we call an epigroup variety \emph{neutral} if it is a neutral element of the lattice \textbf{EPI}. Analogous convention will be applied for all other types of special elements. Our main results give:
\begin{itemize}
\item a complete description of neutral, standard, costandard, distributive and lower-modular varieties;
\item a strong necessary condition and a sufficient condition for modular varieties;
\item a description of commutative modular varieties, strongly permutative codistributive varieties and strongly permutative upper-modular varieties.
\end{itemize}

One can start with formulations of results. We denote by $\mathcal T$, $\mathcal{SL}$ and $\mathcal{ZM}$ the trivial variety, the variety of all semilattices and the variety of all semigroups with zero multiplication respectively. Our first result is the following

\begin{theorem}
\label{neutral and costandard}
For an epigroup variety $\mathcal V$, the following are equivalent:
\begin{itemize}
\item[\textup{a)}] $\mathcal V$ is a neutral element of the lattice $\mathbf{EPI}$;
\item[\textup{b)}] $\mathcal V$ is a costandard element of the lattice $\mathbf{EPI}$;
\item[\textup{c)}] $\mathcal V$ is simultaneously a lower-modular and upper-modular element of the lattice $\mathbf{EPI}$;
\item[\textup{d)}] $\mathcal V$ coincides with one of the varieties $\mathcal T$, $\mathcal{SL}$, $\mathcal{ZM}$ or $\mathcal{SL\vee ZM}$.
\end{itemize}
\end{theorem}

Thus, there are only a few neutral elements in the lattice \textbf{EPI}. In contrast, we note that the lattice of completely regular semigroup varieties contains infinitely many neutral elements including all band varieties, the varieties of all groups, all completely simple semigroups, all orthodox semigroups and some other (this readily follows from~\cite[Corollary~2.9]{Trotter-89}).

Let $\Sigma$ be an identity system written in the language of one associative binary operation and one unary operation. The class of all epigroups that satisfy $\Sigma$ (where the unary operation is treated as pseudoinversion) is denoted by $K_\Sigma$. The class $K_\Sigma$ is not obliged to be a variety because it maybe not closed under taking of (infinite) direct product (see~\cite[Subsection~2.3]{Shevrin-05}, for instance). Note that identity systems $\Sigma$ with the property that $K_\Sigma$ is a variety are completely determined in~\cite[Proposition~2.15]{Gusev-Vernikov-15} (see Lemma~\ref{epigroup variety} below). If $K_\Sigma$ is a variety then we use for this variety the standard notation $\var\Sigma$. It is evident that if the class $K_\Sigma$ consists of periodic epigroups then it is a periodic semigroup variety and therefore, is an epigroup variety. Whence, the notation $\var\Sigma$ is correct in this case. A pair of identities $wx=xw=w$ where the letter $x$ does not occur in the word $w$ is usually written as the symbolic identity $w=0$. This notation is justified because a semigroup with such identities has a zero element and all values of the word $w$ in this semigroup are equal to zero. We will refer to the expression $w=0$ as to a single identity. Such identities and varieties given by them are called 0-\emph{reduced}. Put
\begin{align*}
\mathcal Q&=\var\,\{x^2y=xyx=yx^2=0\},\\
\mathcal Q_n&=\var\,\{x^2y=xyx=yx^2=x_1x_2\cdots x_n=0\},\\
\mathcal R&=\var\,\{x^2=xyx=0\},\\
\mathcal R_n&=\var\,\{x^2=xyx=x_1x_2 \cdots x_n=0\}
\end{align*}
where $n$ is a natural number. We note that $\mathcal Q_1=\mathcal R_1=\mathcal T$. Our second result is the following

\begin{theorem}
\label{distributive and standard}
For an epigroup variety $\mathcal V$, the following are equivalent:
\begin{itemize}
\item[\textup{a)}] $\mathcal V$ is a distributive element of the lattice $\mathbf{EPI}$;
\item[\textup{b)}] $\mathcal V$ is a standard element of the lattice $\mathbf{EPI}$;
\item[\textup{c)}] $\mathcal {V=M \vee N}$ where $\mathcal M$ is one of the varieties $\mathcal T$ or $\mathcal {SL}$, and $\mathcal N$ is one of the varieties $\mathcal Q$, $\mathcal Q_n$, $\mathcal R$ or $\mathcal R_n$.
\end{itemize}
\end{theorem}

The following result gives a complete classification of lower-modular varieties.

\begin{theorem}
\label{lower-modular}
An epigroup variety $\mathcal V$ is a lower-modular element of the lattice $\mathbf{EPI}$ if and only if $\mathcal{V=M\vee N}$ where $\mathcal M$ is one of the varieties $\mathcal T$ or $\mathcal{SL}$, while $\mathcal N$ is a $0$-reduced variety.
\end{theorem} 

We follow the agreement that an adjective indicating a property shared by all semigroups of a given variety is applied to the variety itself; the expressions like `completely regular variety', `periodic variety', `nilvariety', etc.\ are understood in this sense. Recall that an identity of the form
$$x_1x_2\cdots x_n=x_{1\pi}x_{2\pi}\cdots x_{n\pi}$$
where $\pi$ is a non-trivial permutation on the set $\{1,2,\dots,n\}$ is called \emph{permutative}; if $1\pi\ne1$ and $n\pi\ne n$ then this identity is said to be \emph{strongly permutative}. A semigroup or a variety that satisfies [strongly] permutative identity also is called [\emph{strongly}] \emph{permutative}. The following two results describe codistributive and upper-modular varieties in the strongly permutative case.

\begin{theorem}
\label{codistributive}
A strongly permutative epigroup variety $\mathcal V$ is a codistributive element of the lattice $\mathbf{EPI}$ if and only if $\mathcal{V=G\vee X}$ where $\mathcal G$ is an Abelian group variety, while $\mathcal X$ is one of the varieties $\mathcal T$, $\mathcal{SL}$, $\mathcal{ZM}$ or $\mathcal{SL\vee ZM}$.
\end{theorem}

Put $\mathcal C_m=\var\{x^m=x^{m+1},\,xy=yx\}$ for arbitrary natural $m$. In particular, $\mathcal C_1=\mathcal{SL}$. It will be convenient for us also to assume that $\mathcal C_0=\mathcal T$.

\begin{theorem}
\label{upper-modular}
A strongly permutative epigroup variety $\mathcal V$ is an upper-modular element of the lattice $\mathbf{EPI}$ if and only if one of the following holds:
\begin{itemize}
\item[\textup{(i)}] $\mathcal{V=M\vee N}$ where $\mathcal M$ is one of the varieties $\mathcal T$ or $\mathcal{SL}$, and $\mathcal N$ is a nilvariety satisfying the commutative law and the identity
\begin{equation}
\label{xxy=xyy}
x^2y=xy^2;
\end{equation}
\item[\textup{(ii)}] $\mathcal{V=G\vee C}_m\vee\mathcal N$ where $\mathcal G$ is an Abelian group variety, $0\le m\le2$ and $\mathcal N$ satisfies the commutative law and the identity
\begin{equation}
\label{xxy=0}
x^2y=0\ldotp
\end{equation}
\end{itemize}
\end{theorem}

The remaining three results devoted to modular varieties. A word written in the language of multiplication and pseudoinversion is called a \emph{semigroup word} if it does not include the operation of pseudoinversion. An identity is called a \emph{semigroup identity} if both its parts are semigroup words. A semigroup identity $u=v$ is said to be \emph{substitutive} if $u$ and $v$ depend on the same letters and $v$ may be obtained from $u$ by renaming of letters. The following result gives a strong necessary condition for modular varieties. 

\begin{theorem}
\label{modular nec}
If an epigroup variety $\mathcal V$ is a modular element of the lattice $\mathbf{EPI}$ then $\mathcal{V=M\vee N}$ where $\mathcal M$ is one of the varieties $\mathcal T$ or $\mathcal{SL}$, and $\mathcal N$ is a nilvariety given by $0$-reduced and substitutive identities only.
\end{theorem}

It is easy to see that this theorem completely reduces the problem of description of modular varieties to nilvarieties defined by 0-reduced and substitutive identities only (see Corollary~\ref{join with SL or ZM or SL+ZM} below).

We provide the following sufficient condition for modular varieties.

\begin{theorem}
\label{modular suf}
A $0$-reduced semigroup variety is a modular element of the lattice $\mathbf{EPI}$.
\end{theorem}

Theorems~\ref{modular nec} and~\ref{modular suf} show that in order to describe modular varieties we need to examine nil-varieties satisfying substitutive identities. A natural partial case of substitutive identities are permutative ones, while the strongest permutative identity is the commutative law. Modular varieties satisfying this law possess a complete description.

\begin{theorem}
\label{modular commut}
A commutative epigroup variety $\mathcal V$ is a modular element of the lattice $\mathbf{EPI}$ if and only if $\mathcal{V=M\vee N}$ where $\mathcal M$ is one of the varieties $\mathcal T$ or $\mathcal{SL}$ and $\mathcal N$ is a nilvariety that satisfies the commutative law and the identity~\eqref{xxy=0}.
\end{theorem}

Theorems~\ref{modular nec} and~\ref{modular suf} provide a necessary and a sufficient condition for an epigroup variety to be modular respectively. The gap between these conditions seems to be not very large. But the necessary condition is not a sufficient one, while the sufficient condition is not a necessary one. This follows from Theorem~\ref{modular commut}. Indeed, this theorem shows that the variety $\var\{x^3=0,\,xy=yx\}$ is not modular although it is given by 0-reduced and substitutive identities only, while the variety $\var\{x^2y=0,\,xy=yx\}$ is modular although it is not 0-reduced.

The article is structured as follows. It consists of eleven sections. In Section~\ref{preliminaries} we collect definitions, notation and auxiliary results used in what follows. In Section~\ref{SL and ZM are neutral} we verify two special cases of Theorem~\ref{neutral and costandard}, namely the claims that the varieties $\mathcal{SL}$ and $\mathcal{ZM}$ are neutral. These facts are used in each of Sections~\ref{upper-modular proof}--\ref{distributive and standard proof}. After that we prove Theorem~\ref{upper-modular} in Section~\ref{upper-modular proof}, Theorem~\ref{codistributive} in Section~\ref{codistributive proof}, Theorems~\ref{modular nec}--\ref{modular commut} in Section~\ref{modular proof} and Theorem~\ref{lower-modular} in Section~\ref{lower-modular proof}. In Section~\ref{definafle} we discuss an interesting application of Theorems~\ref{neutral and costandard} and~\ref{lower-modular} concerning so-called definable sets of varieties (a corresponding definition see in Section~\ref{definafle}). In Sections~\ref{neutral and costandard proof} and~\ref{distributive and standard proof} we prove Theorems~\ref{neutral and costandard} and~\ref{distributive and standard} respectively. Sections~\ref{upper-modular proof}--\ref{lower-modular proof},~\ref{neutral and costandard proof} and~\ref{distributive and standard proof} contain also a number of corollaries of main results. In particular, we note that an epigroup variety is modular whenever it is lower-modular (Corollary \ref{lower-modular implies modular}). Finally, in Section~\ref{questions} we formulate several open questions.

\section{Preliminaries}
\label{preliminaries}

\subsection{Some properties of the operation of pseudoinversion}
\label{preliminaries pseudoinversion}

The following three lemmas are well known and may be easily checked.

\begin{lemma}
\label{identity for cr}
The identity
\begin{equation}
\label{x=x**}
x=\,\overline{\overline x}
\end{equation}
holds in an epigroup $S$ if and only if $S$ is completely regular.\qed
\end{lemma}

It is well known (see~\cite{Shevrin-94,Shevrin-05}, for instance) that if $S$ is an epigroup and $x\in S$ then $x\,\overline x\,=\,\overline x\,x=x^\omega$. This permits to write in epigroup identities expressions of the form $u^\omega$ rather than $u\,\overline u$, for brevity.

\begin{lemma}
\label{identity for comb}
If an epigroup variety $\mathcal V$ satisfies the identity $x^m=x^{m+1}$ then the identities $x^\omega=\,\overline x=\,\overline{\overline x}\,=x^m$ hold in $\mathcal V$.\qed
\end{lemma}

\begin{lemma}
\label{identity for nil}
The identity
\begin{equation}
\label{x*=0}
\overline x\,=0
\end{equation}
holds in an epigroup $S$ if and only if $S$ is a nil-semi\-group.\qed
\end{lemma}

\subsection{When $K_\Sigma$ is a variety?}

An identity is called \emph{mixed} if one of its parts is a semigroup word but another one is not. Further, a semigroup identity is called \emph{balanced} if each letter occurs in both its parts the same number of times. 

\begin{lemma}[\!\!\,{\mdseries\cite[Proposition~2.15]{Gusev-Vernikov-15}}]
\label{epigroup variety}
The class $K_\Sigma$ is an epigroup variety if and only if $\Sigma$ contains either a semigroup non-balanced identity or a mixed identity.\qed
\end{lemma} 

\subsection{Identities of certain varieties}
\label{preliminaries identities}

We denote by $F$ the free unary semigroup over a countably infinite alphabet (with the operations $\cdot$ and $\overline{\phantom x}$\,). Elements of $F$ are called \emph{words}. If $w\in F$ then we denote by $c(w)$ the set of all letters occurring in $w$ and by $t(w)$ the last letter of $w$. A letter is called \emph{simple} [\emph{multiple}] \emph{in a word} $w$ if it occurs in $w$ ones [at least twice]. Put
$$\mathcal P=\var\{xy=x^2y,\,x^2y^2=y^2x^2\}\quad\text{and}\quad\overleftarrow{\mathcal P}=\var\{xy=xy^2,\,x^2y^2=y^2x^2\}\ldotp$$
The symbol $\equiv$ stands for the equality relation on the unary semigroup $F$. The first two claims of the following lemma are well-known and may be easily verified, the third one was proved in~\cite[Lemma~7]{Golubov-Sapir-82}.

\begin{lemma}
\label{word problem}
A non-trivial semigroup identity $v=w$ holds:
\begin{itemize}
\item[\textup{(i)}] in the variety $\mathcal{SL}$ if and only if $c(v)=c(w)$;
\item[\textup{(ii)}] in the variety $\mathcal C_2$ if and only if $c(v)=c(w)$ and every letter from $c(v)$ is either simple both in $v$ and $w$ or multiple both in $v$ and $w$;
\item[\textup{(iii)}] in the variety $\mathcal P$ if and only if $c(v)=c(w)$ and either the letters $t(v)$ and $t(w)$ are multiple in $v$ and $w$ respectively or $t(v)\equiv t(w)$ and the letter $t(v)$ is simple both in $v$ and $w$.\qed
\end{itemize}
\end{lemma}

If $w$ is a semigroup word then $\ell(w)$ stands for the length of $w$; otherwise, we put $\ell(w)=\infty$. We need the following three remarks about identities of nil-semi\-groups. 

\begin{lemma}
\label{splitting}
Let $\mathcal V$ be a nilvariety.
\begin{itemize}
\item[\textup{(i)}] If the variety $\mathcal V$ satisfies an identity $u=v$ with $c(u)\ne c(v)$ then $\mathcal V$ satisfies also the identity $u=0$.
\item[\textup{(ii)}] If the variety $\mathcal V$ satisfies an identity of the form $u=vuw$ where the word $vw$ is non-empty then $\mathcal V$ satisfies also the identity $u=0$.
\item[\textup{(iii)}] If the variety $\mathcal V$ satisfies an identity of the form $x_1x_2\cdots x_n=v$ and $\ell(v)\ne n$ then $\mathcal V$ satisfies also the identity $x_1x_2\cdots x_n=0$.
\end{itemize}
\end{lemma}

\begin{proof}
The claims~(i) and~(ii) are well known and easily verified.

\smallskip

(iii) If $v$ is a non-semi\-group word, it suffices to refer to Lemma~\ref{identity for nil}. Let now $v$ be a semigroup word. If $\ell(v)<n$ then $c(v)\ne\{x_1,x_2,\dots,x_n\}$, and the desired conclusion follows from the claim~(i). Finally, if $\ell(v)>n$ then the claim we prove readily follows from~\cite[Lemma~1]{Sapir-Sukhanov-81}.
\end{proof}

\subsection{Decomposition of some varieties into the join of subvarieties}
\label{preliminaries decomposition}

As usual, we denote by $\Gr S$ the set of all group elements of an epigroup $S$. For an arbitrary epigroup variety $\mathcal X$, we put $\mathcal{\Gr(X)=X\wedge GR}$ where $\mathcal{GR}$ is the variety of all groups. The variety generated by an epigroup $S$ is denoted by $\var\,S$. Put
$$\mathcal{LZ}=\var\{xy=x\}\quad\text{and}\quad\mathcal{RZ}=\var\{xy=y\}\ldotp$$
The following two facts play an important role in the proof of Theorem~\ref{upper-modular}. `Semigroup prototypes' of Proposition~\ref{decomposition} and Lemma~\ref{monoid decomposition} were given in~\cite[Proposition~1]{Volkov-89} and~\cite{Head-68} respectively.

\begin{proposition}
\label{decomposition}
If $\mathcal V$ is an epigroup variety and $\mathcal V$ does not contain the varieties $\mathcal{LZ}$, $\mathcal{RZ}$, $\mathcal P$ and $\overleftarrow{\mathcal P}$ then $\mathcal{V=M\vee N}$ where $\mathcal M$ is a variety generated by a monoid, and $\mathcal N$ is a nilvariety.
\end{proposition}

\begin{proof}
It is verified in~\cite[Lemma~2]{Volkov-89} that if a semigroup variety does not contain the varieties $\mathcal{LZ}$, $\mathcal{RZ}$, $\mathcal P$ and $\overleftarrow{\mathcal P}$ then $\mathcal V$ satisfies the quasiidentity
\begin{equation}
\label{center idemp}
e^2=e\longrightarrow ex=xe\ldotp
\end{equation}
Repeating literally the proof of this claim (with using a term `subepigroup' rather than `subsemigroup'), one can establish that the similar claim is true for epigroup varieties. Thus, $\mathcal V$ satisfies the quasiidentity~\eqref{center idemp}. The rest of the proof is quite similar to the proof of Proposition~1 in~\cite{Volkov-89}.

Let $S$ be an epigroup that generates the variety $\mathcal V$, $x\in S$ and $E$ the set of all idempotents from $S$. In view of~\eqref{center idemp}, $ES$ is an ideal in $S$. By the definition of an epigroup, there is a natural $n$ such that $x^n\in\Gr\,S$. Then $x^n=x^\omega x^n$ and $x^\omega\in E$. We see that $x^n\in ES$. Therefore, the Rees quotient semigroup $S/ES$ is a nil-semi\-group and therefore, is an epigroup. The natural homomorphism $\rho$ from $S$ onto $S/ES$ separates elements from $S\setminus ES$.

Let now $e\in E$. In view of~\eqref{center idemp}, we have that $eS$ is a subsemigroup in $S$. It is well known that every epigroup satisfies the identity $\overline x\,=x\,\overline x\,^2$ (see~\cite{Shevrin-94,Shevrin-05}, for instance). Hence the equality $\overline{ex}\,=ex\bigl(\,\overline{ex}\,\bigr)^2$ holds. We have verified that, for any $e\in E$, the set $eS$ is a subepigroup in $S$. Put $S^\ast=\prod\limits_{e\in E}eS$. Then $S^\ast$ is an epigroup with unit $(\dots,e,\dots)_{e\in E}$. It follows from~\eqref{center idemp} that the map $\varepsilon$ from $S$ into $S^\ast$ given by the rule $\varepsilon(x)=(\dots,ex,\dots)_{e\in E}$ is a semigroup homomorphism. As is well known (see~\cite{Shevrin-94,Shevrin-05}, for instance), an arbitrary semigroup homomorphism $\xi$ from an epigroup $S_1$ into an epigroup $S_2$ is also an epigroup homomorphism (i.\,e., $\xi\bigl(\,\overline a\,\bigr)=\,\overline{\xi(a)}$ for any $a\in S_1$). Therefore, $\varepsilon$ is an epigroup homomorphism from $S$ into $S^\ast$. One can verify that $\varepsilon$ separates elements of $ES$. Let $e,f\in E$, $x,y\in S$ and $\varepsilon(ex)=\varepsilon(fy)$. Then $e\cdot ex=e\cdot fy$ and $f\cdot ex=f\cdot fy$. Since $e,f\in E$, we have
\begin{equation}
\label{ex=efy,fex=fy}
ex=efy\quad\text{and}\quad fex=fy\ldotp
\end{equation}
Therefore,
$$ex\stackrel{\eqref{ex=efy,fex=fy}}=efy\stackrel{\eqref{ex=efy,fex=fy}}=efex\stackrel{\eqref{center idemp}}=feex\stackrel{e\in E}=fex\stackrel{\eqref{ex=efy,fex=fy}}=fy\ldotp$$
We see that $ex=fy$ whenever $\varepsilon(ex)=\varepsilon(fy)$. This means that $\varepsilon$ separates elements of $ES$.

Thus, $\varepsilon$ and $\rho$ are homomorphisms from $S$ into $S^\ast$ and $S/ES$ respectively, and the intersection of kernels of these homomorphisms is trivial. Therefore, the epigroup $S$ is decomposable into a subdirect product of the epigroups $S^\ast$ and $S/ES$, whence $\mathcal{V\subseteq M\vee N}$ where $\mathcal M=\var\,S^\ast$ is a variety generated by a monoid and $\mathcal N=\var(S/ES)$ is a nilvariety. On the other hand $S^\ast,S/ES\in\mathcal V$, whence $\mathcal{M\vee N\subseteq V}$. We have proved that $\mathcal{V=M\vee N}$.
\end{proof}

Recall that an epigroup is called \emph{combinatorial} if all its subgroups are trivial. 

\begin{lemma}
\label{monoid decomposition}
If an epigroup variety $\mathcal M$ is generated by a commutative epigroup with unit then $\mathcal{M=G\vee C}_m$ for some Abelian group variety $\mathcal G$ and some $m\ge0$.
\end{lemma}

\begin{proof}
It is well known that the variety of all Abelian groups is the least non-periodic epigroup variety. This variety evidently contains the infinite cyclic group. Further, for each natural $m$, let $G_m$ denote the cyclic group of order $m$. It is evident that if $\mathcal M$ is periodic then the set $\{m\in\mathbb N\mid G_m\in\mathcal M\}$ has the greatest element. We denote by $G$ the infinite cyclic group whenever the variety $\mathcal M$ is non-periodic, and the finite cyclic group of the greatest order among all cyclic groups in $\mathcal M$ otherwise. In both the cases $G\in\mathcal M$. Further, let $D_m$ be the finite cyclic combinatorial epigroup of order $m$ and $d_m$ is a generator of $D_m$. Put $X=\{m\in\mathbb N\mid D_m\in\mathcal M\}$. If the set $X$ has not the greatest element then the semigroup $\prod\limits_{m\in X}D_m$ is not an epigroup since, for example, no power of the element $(\dots,d_m,\dots)_{m\in X}$ belongs to a subgroup. Therefore, the set of numbers $X$ contains the greatest number. We denote this number by $n$. Repeating literally arguments from the proof of Theorem~1 in~\cite{Head-68}, we have that every epigroup from $\mathcal M$ is a homomorphic image of some subepigroup of the epigroup $G\times D_n$. Therefore, $\mathcal{M=G\vee D}$ where $\mathcal G=\var\,G$ and $\mathcal D=\var\,D_n$. Clearly, $\mathcal G$ is a variety of Abelian groups. The variety $\mathcal D$ is generated by a finite epigroup, whence it may be considered as a semigroup variety. It is well known and may be easily verified that $(m+1)$-element combinatorial cyclic monoid generates the variety $\mathcal C_m$. Therefore, $\mathcal{D=C}_m$ for some $m\ge0$.
\end{proof}

It is evident that a strongly permutative variety does not contain the varieties $\mathcal{LZ}$, $\mathcal{RZ}$, $\mathcal P$ and $\overleftarrow{\mathcal P}$. Besides that, every monoid satisfying a permutative identity is commutative. Thus, we have the following corollary of Proposition~\ref{decomposition} and Lemma~\ref{monoid decomposition}.

\begin{corollary}
\label{str permut decomposition}
If $\mathcal V$ is a strongly permutative epigroup variety then $\mathcal{V=\mathcal G\vee C}_m\vee\mathcal N$ where $\mathcal G$ is an Abelian group variety, $m\ge0$ and $\mathcal N$ is a nilvariety.\qed{\sloppy

}
\end{corollary}

\subsection{A direct decomposition of one varietal lattice}
\label{preliminaries direct}

We denote by $\mathcal{AG}$ the variety of all Abelian groups. The aim of this subsection is to prove the following

\begin{proposition}
\label{direct product}
The lattice $L(\mathcal{AG\vee C}_2\vee\mathcal Q)$ is isomorphic to the direct product of the lattices $L(\mathcal{AG})$ and $L(\mathcal C_2\vee\mathcal Q)$.
\end{proposition}

\begin{proof}
We need the following auxiliary statement.

\begin{lemma}
\label{within AG+C_2+Q}
If $\mathcal{X\subseteq AG\vee C}_2\vee\mathcal Q$ then $\mathcal{X=G\vee C}_m\vee\mathcal N$ where $\mathcal G$ is some Abelian group variety, $0\le m\le2$ and $\mathcal{N\subseteq Q}$.
\end{lemma}

\begin{proof}
Being a subvariety of the variety $\mathcal{AG\vee C}_2\vee\mathcal Q$, the variety $\mathcal X$ satisfies the identity $x^2y=yx^2$. It is evident that this identity fails in the varieties $\mathcal{LZ}$ and $\mathcal{RZ}$. Further, Lemma~\ref{word problem}(iii) and the dual statement imply that this identity is false in the varieties $\mathcal P$ and $\overleftarrow{\mathcal P}$ as well. Therefore, none of the four mentioned varieties is contained in $\mathcal X$. Besides that, the variety $\mathcal{AG\vee C}_2\vee\mathcal Q$ (and therefore, $\mathcal X$) satisfies the identity $x^2yz=x^2zy$. Substituting~1 for $x$, we have that all monoids in $\mathcal X$ are commutative. Proposition~\ref{decomposition} and Lemma~\ref{monoid decomposition} imply now that $\mathcal{X=G\vee C}_m\vee\mathcal N$ for some Abelian group variety $\mathcal G$, some $m\ge0$ and some nilvariety $\mathcal N$. It is evident that $\mathcal{G\subseteq AG}$. Lemmas~\ref{identity for cr},~\ref{identity for comb} and~\ref{identity for nil} imply that $\mathcal{AG}$, $\mathcal C_m$ and $\mathcal Q$ satisfy the identities~\eqref{x=x**},~$\overline{\overline x}\,=x^m$ and~\eqref{x*=0} respectively. Therefore, the identity $x^2y=\,\overline{\overline x}\,^2y$ holds in the variety $\mathcal{AG\vee C}_2\vee\mathcal Q$. But this identity is false in the variety $\mathcal C_m$ with $m>2$. Hence $m\le2$. This implies that the variety $\mathcal{AG\vee C}_m\vee\mathcal Q$ satisfies the identities $x^2y=yx^2=\,\overline{\overline x}\,^2y$ and $xyx=xy\,\overline{\overline x}$. Since $\mathcal{N\subseteq X\subseteq AG\vee C}_2\vee\mathcal Q$, Lemma~\ref{identity for nil} implies that the variety $\mathcal N$ satisfies the identities $x^2y=xyx=yx^2=0$, whence $\mathcal{N\subseteq Q}$.
\end{proof}

A semigroup word $w$ is called \emph{linear} if every letter from $c(w)$ is simple in $w$. For convenience of references, we formulate the following observation that will be helpful also in Section~\ref{distributive and standard proof}.

\begin{lemma}
\label{what =0 in Q}
All non-semi\-group words and all non-linear semigroup words except $x^2$ equal to~$0$ in the variety $\mathcal Q$.
\end{lemma}

\begin{proof}
If $u$ is a non-semi\-group word then $u=0$ in $\mathcal Q$ by Lemma~\ref{identity for nil}. The claim that a non-linear semigroup word differ from $x^2$ equal to~0 in $\mathcal Q$ is evident.
\end{proof}

Now we start with the direct proof of Proposition~\ref{direct product}. Let $\mathcal{V\subseteq AG\vee C}_2\vee\mathcal Q$. In view of Lemma~\ref{within AG+C_2+Q}, $\mathcal{V=G\vee C}_m\vee\mathcal N$ for some Abelian group variety $\mathcal G$, some $0\le m\le2$ and some variety $\mathcal N$ with $\mathcal{N\subseteq Q}$. Put $\mathcal{U=C}_m\vee\mathcal N$. We have that $\mathcal{V=G\vee U}$ where $\mathcal{G\subseteq AG}$ and $\mathcal{U\subseteq C}_2\vee\mathcal Q$. It remains to establish that this decomposition of the variety $\mathcal V$ into the join of some subvariety of the variety $\mathcal{AG}$ and some subvariety of the variety $\mathcal C_2\vee\mathcal Q$ is unique.

Let $\mathcal{V=G'\vee U'}$ where $\mathcal{G'\subseteq AG}$ and $\mathcal{U'\subseteq C}_2\vee\mathcal Q$. We need to verify that $\mathcal{G=G'}$ and $\mathcal{U=U'}$. Let $u=v$ be an arbitrary identity satisfied by $\mathcal G$. The variety $\mathcal U$ satisfies the identity $x^3=x^4$. Then the identity $u^4v^3=u^3v^4$ holds in the variety $\mathcal{G\vee U}$. Let us cancel this identity on $u^3$ from the left and on $v^3$ from the right, thus concluding that $u=v$ holds in $\mathcal{\Gr\,(G\vee U)}$. Therefore, $\mathcal{\Gr\,(G\vee U)\subseteq G}$. The opposite inclusion is evident. Thus, $\mathcal{\Gr\,(G\vee U)=G}$. Analogously, $\mathcal{\Gr\,(G'\vee U')=G'}$. We see that
$$\mathcal{G=\Gr\,(G\vee U)=\Gr\,(V)=\Gr\,(G'\vee U')=G'}\ldotp$$
It remains to check that $\mathcal{U=U'}$. Recall that $\mathcal{U=C}_m\vee\mathcal N$ where $0\le m\le2$ and $\mathcal{N\subseteq Q}$, while $\mathcal{U'\subseteq C}_2\vee\mathcal Q$. It is evident that the variety $\mathcal C_2\vee\mathcal Q$ (and therefore, $\mathcal U'$) is combinatorial. Therefore, Lemma~\ref{within AG+C_2+Q} implies that $\mathcal{U'=C}_k\vee\mathcal N'$ for some $0\le k\le2$ and some variety $\mathcal N'$ with $\mathcal{N'\subseteq Q}$.

Suppose that $m\ne k$. We may assume without any loss that $m<k$, i.\,e., either $m=0$, $1\le k\le2$ or $m=1$, $k=2$. Suppose at first that $m=0$ and $1\le k\le2$. It is evident that any group satisfies the identity $x^\omega=y^\omega$. Lemma~\ref{identity for nil} implies that this identity holds in $\mathcal N$ and therefore, in $\mathcal{V=G\vee T\vee N=G\vee N}$. But Lemma~\ref{word problem}(ii) implies that this identity fails in the variety $\mathcal{SL}$. However, this is impossible because $\mathcal{SL\subseteq G'\vee C}_k\vee\mathcal{N'=V}$. Suppose now that $m=1$ and $k=2$. Then Lemmas~\ref{identity for cr},~\ref{identity for comb} and~\ref{identity for nil} imply that the identity $x^2y=x^2\,\overline{\overline y}$ holds in the variety $\mathcal{V=G\vee C}_m\vee\mathcal{N=G\vee SL\vee N}$. But this identity is false in the variety $\mathcal C_2$ (and therefore, in $\mathcal{G'\vee C}_k\vee\mathcal{N'=V}$) by Lemmas~\ref{identity for comb} and~\ref{word problem}(ii). A contradiction shows that $m=k$.

Note that if $\mathcal X$ is an arbitrary epigroup variety then the class of all epigroups in $\mathcal X$ satisfying the identity~\eqref{x*=0} is the greatest nilsubvariety of $\mathcal X$. We denote this subvariety by $\Nil(\mathcal X)$. Put $\overline{\mathcal N}=\Nil\,(\mathcal U)$ and $\overline{\mathcal N'}=\Nil\,(\mathcal U')$. It suffices to verify that $\overline{\mathcal N}=\overline{\mathcal N'}$ because
$$\mathcal{U=C}_m\vee\mathcal{N=C}_m\vee\overline{\mathcal N}=\mathcal C_k\vee\overline{\mathcal N'}=\mathcal C_k\vee\mathcal{N'=U'}$$
in this case. Suppose that $\overline{\mathcal N}\ne\overline{\mathcal N'}$. It suffices to verify that
\begin{equation}
\label{inequality}
\mathcal{G\vee C}_m\vee\overline{\mathcal N}\ne\mathcal{G'\vee C}_k\vee\overline{\mathcal N'}
\end{equation}
because this contradicts the equalities
$$\mathcal{G\vee C}_m\vee\overline{\mathcal N}=\mathcal{G\vee C}_m\vee\mathcal{N=V=G'\vee C}_k\vee\mathcal{N'=G'\vee C}_k\vee\overline{\mathcal N'}\ldotp$$

We will say that varieties $\mathcal X_1$ and $\mathcal X_2$ \emph{differ with an identity} $u=v$ if this identity holds in one of the varieties $\mathcal X_1$ or $\mathcal X_2$ but fails in another one. Since $\overline{\mathcal N}\ne\overline{\mathcal N'}$, the varieties $\overline{\mathcal N}$ and $\overline{\mathcal N'}$ differ with some identity. We may assume without loss of generality that this identity holds in $\overline{\mathcal N}$ but fails in $\overline{\mathcal N'}$. Suppose at first that the identity we mention is a 0-reduced identity $u=0$. Since $\overline{\mathcal N'}\subseteq\mathcal Q$, this identity fails in $\mathcal Q$. In view of Lemma~\ref{what =0 in Q}, either $u$ is a linear word or $u\equiv x^2$. Suppose that $u\equiv x_1x_2\cdots x_n$ for some $n$. Then the identity $x_1x_2\dots x_n=0$ holds in $\overline{\mathcal N}$ but fails in $\overline{\mathcal N'}$. If $m=2$ then the variety $\overline{\mathcal N}$ contains the variety $\Nil\,(\mathcal C_2)=\var\{x^2=0,\,xy=yx\}$. The latest variety does not satisfy the identity $x_1x_2\dots x_n=0$. Therefore, $m\le1$. Then the variety $\mathcal{G\vee C}_m$ is completely regular. Lemmas~\ref{identity for cr} and~\ref{identity for nil} imply now that the variety $\mathcal{G\vee C}_m\vee\overline{\mathcal N}$ satisfies the identity $x_1x_2\cdots x_n=\,\overline{\overline{x_1}}\,x_2\cdots x_n$. But Lemma~\ref{identity for nil} implies that this identity is false in $\overline{\mathcal N'}$ and therefore, in $\mathcal{G'\vee C}_k\vee\overline{\mathcal N'}$. Thus,~\eqref{inequality} holds. Let now $u\equiv x^2$. Then Lemmas~\ref{identity for cr},~\ref{identity for comb} and~\ref{identity for nil} imply that the identity $x^2=\,\overline{\overline x}\,^2$ holds in the variety $\mathcal{G\vee C}_m\vee\overline{\mathcal N}$ but is false in the variety $\mathcal{G'\vee C}_k\vee\overline{\mathcal N'}$. We see again that the inequality~\eqref{inequality} holds.

It remains to consider the case when $\overline{\mathcal N}$ and $\overline{\mathcal N'}$ differ with some non-0-reduced identity $u=v$. Suppose that $c(u)\ne c(v)$. Lemma~\ref{splitting}(i) then implies that the variety $\overline{\mathcal N}$ satisfies both the identities $u=0$ and $v=0$. Then the variety $\overline{\mathcal N'}$ does not satisfy at least one of them because $\overline{\mathcal N}$ and $\overline{\mathcal N'}$ do not differ with $u=v$ otherwise. We see that $\overline{\mathcal N}$ and $\overline{\mathcal N'}$ differ with some 0-reduced identity, and we go to the situation considered in the previous paragraph. Let now $c(u)=c(v)$. Suppose that the identity $u=0$ holds in $\mathcal Q$. Since $\overline{\mathcal N},\overline{\mathcal N'}\subseteq\mathcal Q$, we have that the identity $u=0$ holds in both the varieties $\overline{\mathcal N}$ and $\overline{\mathcal N'}$. Then $\overline{\mathcal N}$ satisfies also the identity $v=0$. But $v=0$ fails in $\overline{\mathcal N'}$ because $u=v$ holds in $\overline{\mathcal N'}$ otherwise. We see that the varieties $\overline{\mathcal N}$ and $\overline{\mathcal N'}$ differ with some 0-reduced identity. This case has been already considered in the previous paragraph. Thus the identity $u=0$ fails in $\mathcal Q$. Analogously, $v=0$ fails in $\mathcal Q$. Since the identity $u=v$ is non-trivial and $c(u)=c(v)$, Lemma~\ref{what =0 in Q} implies that the words $u$ and $v$ are linear. Using the fact that $c(u)=c(v)$ again, we have that the identity $u=v$ is permutative. Then it is evident that this identity holds in $\mathcal{G\vee C}_m\vee\overline{\mathcal N}$ but fails in $\mathcal{G'\vee C}_k\vee\overline{\mathcal N'}$. We prove that the inequality~\eqref{inequality} holds.

Thus,~\eqref{inequality} fulfills always, whence we are done.
\end{proof}

Analog of Proposition~\ref{direct product} for semigroup varieties was proved in~\cite[Proposition~2a]{Vernikov-88} (namely it was checked there that $L(\mathcal{G\vee C}_2\vee\mathcal Q)\cong L(\mathcal G)\times L(\mathcal C_2\vee\mathcal Q)$ where $\mathcal G$ is an Abelian periodic group variety). The proof of Proposition~\ref{direct product} given above is quite similar to the proof of the mentioned result from~\cite{Vernikov-88}. But results of~\cite{Vernikov-88} was not used directly above. Therefore, the mentioned result from~\cite{Vernikov-88} may be considered now as a consequence of Proposition~\ref{direct product}.

\subsection{Varieties of finite degree}
\label{preliminaries degree}

If $n$ is a natural number then a variety $\mathcal X$ is called a \emph{variety of degree} $n$ if all nil-semi\-groups in $\mathcal X$ are nilpotent of degree $\le n$ and $n$ is the least number with such a property. If $\mathcal X$ is not a variety of degree $\le n$, we will say that $\mathcal X$ is a variety of \emph{degree} $>n$. A variety is said to be a \emph{variety of finite degree} if it is a variety of degree $n$ for some $n$. If $\mathcal V$ is a variety of finite degree, we denote the degree of $\mathcal V$ by $\deg(\mathcal V)$; otherwise we put $\deg(\mathcal V)=\infty$. We need the following

\begin{proposition}[\!\!{\mdseries\cite[Corollary~1.3]{Gusev-Vernikov-15}}]
\label{degree}
Let $n$ be an arbitrary natural number. For an epigroup variety $\mathcal V$, the following are equivalent:
\begin{itemize}
\item[$1)$] $\deg(\mathcal V)\le n$;
\item[$2)$] $\mathcal V\nsupseteq\var\{x^2=x_1x_2\cdots x_{n+1}=0,\,xy=yx\}$;
\item[$3)$] $\mathcal V$ satisfies an identity of the form
\begin{equation}
\label{equ-fin-deg}
x_1\cdots x_n=x_1\cdots x_{i-1}\cdot\overline{\overline{x_i\cdots x_j}}\cdot x_{j+1}\cdots x_n
\end{equation}
for some $i$ and $j$ with $1\le i\le j\le n$.\qed
\end{itemize}
\end{proposition}

Proposition~\ref{degree} readily implies

\begin{corollary}
\label{degree of meet}
$\mathcal{\deg(X\wedge Y)=\min\bigl\{\deg(X),\deg(Y)\bigr\}}$ for arbitrary epigroup varieties $\mathcal X$ and $\mathcal Y$.\qed
\end{corollary}

The following corollary may be proved quite analogously to Corollary~2.13 of~\cite{Vernikov-08b} with referring to Proposition~\ref{degree} rather than Proposition~2.11 of~\cite{Vernikov-08b}.

\begin{corollary}
\label{degree of join with nil}
If $\mathcal V$ is an arbitrary epigroup variety and $\mathcal N$ is a nilvariety then $\mathcal{\deg(V\vee N)=\max\bigl\{\deg(V),\deg(N)\bigr\}}$.\qed
\end{corollary}

Note that the analog of Corollary~\ref{degree of join with nil} for arbitrary epigroup varieties is wrong even in the periodic case. For instance, it is easy to deduce from Lemma~\ref{word problem}(iii), the dual fact and Proposition~\ref{degree} that $\deg(\mathcal P)=\deg\bigl(\overleftarrow{\mathcal P}\bigr)=2$ but $\deg\bigl(\mathcal P\vee\overleftarrow{\mathcal P}\bigr)=3$.

Proposition~\ref{degree} and Lemma~\ref{identity for cr} easily imply

\begin{corollary}
\label{degree of join with cr}
If $\mathcal V$ is an arbitrary epigroup variety and $\mathcal X$ is a completely regular variety then $\mathcal{\deg(V\vee X)=\deg(V)}$.\qed
\end{corollary}

\subsection{Some properties of special elements of lattices}
\label{preliminaries special}

The following claim is well known. 

\begin{lemma}[\!\!{\mdseries\cite[Lemma~II.1.1]{Gratzer-Schmidt-61}}]
\label{distr+mod=stand}
If an element of a lattice $L$ is distributive and modular in $L$ then it is standard in $L$.\qed
\end{lemma}

This fact together with~\cite[Theorem~255(iii)]{Gratzer-11} imply the following

\begin{lemma}
\label{neutral=distr+codistr+mod}
An element of a lattice $L$ is neutral in $L$ if and only if it is distributive, codistributive and modular in $L$.\qed
\end{lemma}

Let $I$ be a lattice identity of the form $s=t$ where $s$ and $t$ are lattice terms depending on ordering set of variables $x_0,x_1,\dots,x_n$. An element $x$ of a lattice $L$ is called an $I$-\emph{element}\footnote{Probably, it would be more correct to say about $(I,\sigma)$-\emph{elements} where $\sigma=(x_0,x_1,\dots,x_n)$ is a permutation on the set of variables occurring in $I$.} if $s(x,x_1,\dots,x_n)=t(x,x_1,\dots,x_n)$ for all $x_1,\dots,x_n\in L$. Special elements of all types mentioned above are $I$-elements for appropriate identities $I$ and ordering of variables occurring in $I$ (this claim is evident for neutral, [co]standard and [co]distributive elements, and follows from the fact that the modular law may be written in the form of identity for modular, upper-modular and lower-modular elements).

\begin{lemma}[\!\!{\mdseries\cite[Corollary~2.1]{Shaprynskii-11}}]
\label{join with neutral atom}
Let $I$ be a non-trivial lattice identity, $L$ a lattice with $0$, $x\in L$ and $a$ an atom and a neutral element of the lattice $L$. Then $x$ is an $I$-element in $L$ if and only if $x\vee a$ has the same property.\qed
\end{lemma}

\subsection{$\mathcal{SL}$ and $\mathcal{ZM}$ are atoms}
\label{preliminaries atoms}

It is evident that atoms of the lattice \textbf{EPI} coincide with atoms of the lattice \textbf{SEM}. The list of atoms of the latter lattice is generally known (see~\cite{Shevrin-Vernikov-Volkov-09}, for instance). In particular, the following is valid.

\begin{lemma}
\label{SL and ZM are atoms}
The varieties $\mathcal{SL}$ and $\mathcal{ZM}$ are atoms of the lattice $\mathbf{EPI}$.\qed
\end{lemma}

\section{The varieties $\mathcal{SL}$ and $\mathcal{ZM}$ are neutral}
\label{SL and ZM are neutral}

\begin{proposition}
\label{SL is neutral}
The variety $\mathcal{SL}$ is a neutral element of the lattice $\mathbf{EPI}$.
\end{proposition}

\begin{proof}
In view of Lemma~\ref{neutral=distr+codistr+mod}, it suffices to verify that the variety $\mathcal{SL}$ is distributive, codistributive and modular. Let $\mathcal X$ and $\mathcal Y$ be arbitrary epigroup varieties.

\smallskip

\emph{Distributivity}. We need to verify the inclusion
$$\mathcal{(SL\vee X)\wedge(SL\vee Y)\subseteq SL\vee(X\wedge Y)}$$
because the opposite inclusion is evident. Suppose that the identity $u=v$ holds in $\mathcal{SL\vee(X\wedge Y)}$. In particular, it holds in $\mathcal{SL}$, whence $c(u)=c(v)$ by Lemma~\ref{word problem}(i). Let $u\equiv w_0,w_1,\dots,w_n\equiv v$ be a deduction of this identity from identities of the varieties $\mathcal X$ and $\mathcal Y$. Further considerations are given by induction on $n$.

\smallskip

\emph{Induction base.} If $n=1$ then the identity $u=v$ holds in one of the varieties $\mathcal X$ or $\mathcal Y$. Whence, it holds in one of the varieties $\mathcal{SL\vee X}$ or $\mathcal{SL\vee Y}$, and therefore $\mathcal{(SL\vee X)\wedge(SL\vee Y)}$ satisfies $u=v$.

\smallskip

\emph{Induction step.} Let now $n>1$. Consider the words $w'_1,\dots,w'_ {n-1}$ obtained from the words $w_1,\dots,w_ {n-1}$ respectively by equating all the letters that are not occur in $u$, for some letter of $c(u)$. Clearly, the sequence of words $u,w'_1,\dots,w'_{n-1},v$ also is a deduction of the identity $u=v$ from identities of the varieties $\mathcal X$ and $\mathcal Y$. Thus, we may assume that $c(w_1),\dots,c(w_{n-1})\subseteq c(u)$. If $c(w_0)=c(w_1)=\cdots=c(w_n)$ then the sequence $w_0,w_1,\dots,w_n$ is a deduction of the identity $u=v$ from identities of the varieties $\mathcal{SL\vee X}$ and $\mathcal{SL\vee Y}$, and we are done. Suppose now that $c(w_k)\ne c(w_{k+1})$ for some $0\le k\le n-1$. Let $i$ be the least index with $c(w_i)\ne c(w_{i+1})$ and $j$ be the greatest index with $c(w_j)\ne c(w_{j-1})$. Suppose that $i>0$. Then $c(w_i)=c(u)=c(v)$ and consequences of words $w_0,w_1,\dots,w_i$ and $w_i,w_{i+1},\dots,w_n$ are deductions of the identities $u=w_i$ and $w_i=v$ respectively from the identities of the varieties $\mathcal X$ and $\mathcal Y$. Lemma~\ref{word problem}(i) implies now that the identities $u=w_i$ and $w_i=v$ hold in the variety $\mathcal{SL\vee(X\wedge Y)}$. By induction assumption these identities hold also in the variety $\mathcal{(SL\vee X)\wedge(SL\vee Y)}$. Whence, the last variety satisfies the identity $u=v$ too. The case when $j<n$ may be considered quite analogously. Thus, we may suppose that $i=0$ and $j=n$. In other words, $c(u)\ne c(w_1)$ and $c(v)\ne c(w_{n-1})$.

The identity $u=w_1$ holds in one of the varieties $\mathcal X$ and $\mathcal Y$. Suppose that it holds in $\mathcal X$. Since $c(u)\ne c(w_1)$, Lemma~\ref{word problem}(i) implies that $\mathcal{SL\nsubseteq X}$. Let $S$ be an epigroup in $\mathcal X$ and $\zeta$ a homomorphism from $F$ to $S$. For a word $w$, we denote by $w^\zeta$ the image of $w$ under $\zeta$. It is well known (see~\cite{Shevrin-94,Shevrin-05}, for instance) that a variety that does not contain $\mathcal{SL}$ consists of archimedean epigroups. Further, a set of group elements in an archimedean epigroup is an ideal of this epigroup. In particular, this is the case for the epigroup $S$. Now we are going to check that $u^\zeta\in\Gr S$. Since $c(w_1)\subset c(u)$, there is a letter $x\in c(u)\setminus c(w_1)$. Substituting $x^\omega$ for $x$ in the identity $u=w_1$, we obtain the identity $u'=w_1$ that holds in $\mathcal X$. Therefore, $\mathcal X$ satisfies the identity $u=u'$. The word $u'$ contains a subword $x^\omega$. Since $(x^\omega)^\zeta\in\Gr S$ and $\Gr S$ is an ideal in $S$, we have that $u^\zeta\in\Gr S$. Therefore, $\mathcal X$ satisfies the identity $u=uu^\omega$. Similar arguments show that the identity $u=uu^\omega$ holds in $\mathcal Y$ whenever $\mathcal Y$ satisfies $u=w_1$, and that one of the varieties $\mathcal X$ and $\mathcal Y$ satisfies the identity $v=vv^\omega$. Therefore, the sequence of words
$$u,uu^\omega,w_1u^\omega,\dots,w_{n-1}u^\omega,vu^\omega,vw_1^\omega,\dots,vw_{n-1}^\omega,vv^\omega,v$$
is a deduction of the identity $u=v$ from identities of the varieties $\mathcal{SL\vee X}$ and $\mathcal{SL\vee Y}$. Hence this identity holds in $\mathcal{(SL\vee X)\wedge(SL\vee Y)}$.

\smallskip

\emph{Codistributivity.} In view of Lemma~\ref{SL and ZM are atoms}, if $\mathcal W$ is an arbitrary epigroup variety then either $\mathcal{W\supseteq SL}$ or $\mathcal{W\wedge SL=T}$. We need to verify that
$$\mathcal{SL\wedge(X\vee Y)=(SL\wedge X)\vee(SL\wedge Y)}\ldotp$$
Clearly, both the parts of this equality coincides with $\mathcal{SL}$ whenever at least one of the varieties $\mathcal X$ or $\mathcal Y$ contains $\mathcal{SL}$. It remains to verify that if $\mathcal{X\nsupseteq SL}$ and $\mathcal{Y\nsupseteq SL}$ then $\mathcal{X\vee Y\nsupseteq SL}$. This claim immediately follows from the fact that there is a non-trivial identity $u=v$ such that an epigroup variety $\mathcal W$ does not contain the variety $\mathcal{SL}$ if and only if $\mathcal W$ satisfies the identity $u=v$ (in particular, the identity $(x^\omega y^\omega x^\omega)^\omega=x^\omega$ has such a property, see~\cite[Corollary~3.2]{Shevrin-05}, for instance).

\smallskip

\emph{Modularity.} Let $\mathcal{X\subseteq Y}$. We need to verify that
$$\mathcal{(SL\vee X)\wedge Y\subseteq(SL\wedge Y)\vee X}$$
because the opposite inclusion is evident. If $\mathcal{SL\subseteq Y}$ then both the parts of the inclusion evidently coincides with $\mathcal{SL\vee X}$. Let now $\mathcal{SL\nsubseteq Y}$. Then Lemma~\ref{SL and ZM are atoms} implies that $\mathcal{SL\wedge Y=T}$, whence $\mathcal{(SL\wedge Y)\vee X=X}$. Suppose that an identity $u=v$ holds in the variety $\mathcal X$. It suffices to verify that this identity holds in $\mathcal{(SL\vee X)\wedge Y}$ too. If $c(u)=c(v)$ then $u=v$ holds in $\mathcal{SL}$ by Lemma~\ref{word problem}(i). Whence it satisfies in $\mathcal{SL\vee X}$, and we are done. Let now $c(u)\ne c(v)$. Since $\mathcal{SL\nsubseteq Y}$, Lemma~\ref{word problem}(i) implies that $\mathcal Y$ satisfies an identity $s=t$ with $c(s)\ne c(t)$. We may assume without any loss that there is a letter $y\in c(t)\setminus c(s)$. Moreover, we may assume that $c(s)=\{x\}$ and $c(t)=\{x,y\}$ (if this is not the case then we equate all letters but $y$ to $x$ in the identity $s=t$ and multiply the identity we obtain on $x$ from the right). Let $s_1$ [respectively $t_1$] be the word obtained from the word $s$ [respectively $t$] by replacing the letters $x$ and $y$ each to other. Clearly, the identity $s_1=t_1$ follows from the identity $s=t$.

Consider the case when the word $v$ may be obtained from $u$ by replacing of one letter to another one. We may assume without loss of generality that we substitute the letter $y$ for the letter $x$. Let $u_1$, $u_2$, $u_3$ and $u_4$ be the words obtained from $u$ by substitution of the words $s$, $s_1$, $t$ and $t_1$ respectively for the letter $x$. Since $x\notin c(v)$, the identities $u_1=v$, $u_2=v$, $u_3=v$ and $u_4=v$ are obtained from $u=v$ by the same substitutions. Therefore, the words $u$, $v$, $u_1$, $u_2$, $u_3$ and $u_4$ are equal each to other in the variety $\mathcal X$. Further, $c(u)=c(u_1)$, $c(v)=c(u_2)$ and $c(u_3)=c(u_4)$, whence the identities $u=u_1$, $v=u_2$ and $u_3=u_4$ hold in $\mathcal{SL\vee X}$ by Lemma~\ref{word problem}(i). The identities $u_1=u_3$ and $u_2=u_4$ follow from $s=t$ and $s_1=t_1$ respectively. Hence these identities hold in $\mathcal Y$. Thus, the sequence of words $u,u_1,u_3,u_4,u_2,v$ is a deduction of the identity $u=v$ from the identities of the varieties $\mathcal{SL\vee X}$ and $\mathcal Y$.

Finally, consider an arbitrary identity $u=v$ that holds in $\mathcal X$. Replacing one by one all the letters from $c(u)\setminus c(v)$ by some letters of $c(v)$, we obtain the sequence $u,w_1,\dots,w_m$, in which any adjacent words differ by replacing one letter. In the sequence of identities $u=v,w_1=v,\dots,w_m=v$, every identity (except the first one) is obtained from the previous one by replacing one letter. Therefore the words $u,w_1,\dots,w_m$ are equal each to other in $\mathcal X$. As we have proved above, this implies that the identities $u=w_1=\cdots=w_m$ hold in $\mathcal{(SL\vee X)\wedge Y}$. Analogously, if we replace in the identity $v=w_m$ all letters from $c(v)\setminus c(w_m)$ by an arbitrary letter from $c(w_m)$, then we obtain a sequence of identities $v=w'_1=\cdots=w'_n$ that hold in $\mathcal{(SL\vee X)\wedge Y}$ as well. In particular, these identities hold in $\mathcal X$. Moreover, since $c(w_m)=c(w'_n)$, Lemma~\ref{word problem}(i) implies that the identity $w_m=w'_n$ holds in $\mathcal{SL\vee X}$. Therefore, the identity $u=v$ holds in $\mathcal{(SL\vee X)\wedge Y}$.
\end{proof}

\begin{proposition}
\label{ZM is neutral}
The variety $\mathcal{ZM}$ is a neutral element of the lattice $\mathbf{EPI}$.
\end{proposition}

\begin{proof}
In view of Lemma~\ref{neutral=distr+codistr+mod}, it suffices to check that $\mathcal{ZM}$ is distributive, codistributive and modular. Let $\mathcal X$ and $\mathcal Y$ be arbitrary epigroup varieties.

\smallskip

\emph{Distributivity.} We need to verify that
$$\mathcal{(ZM\vee X)\wedge(ZM\vee Y)\subseteq ZM\vee(X\wedge Y)}$$
because the opposite inclusion is evident. Suppose that an identity $u=v$ holds in $\mathcal{ZM\vee(X\wedge Y)}$. We aim to check that this identity is satisfied by the variety $\mathcal{(ZM\vee X)\wedge(ZM\vee Y)}$. The identity $u=v$ holds in $\mathcal{ZM}$ and there is a deduction of this identity from identities of the varieties $\mathcal X$ and $\mathcal Y$. In other words, there are words $u_0,u_1,\dots,u_n$ such that $u_0\equiv u$, $u_n\equiv v$ and, for each $i=0,1,\dots,n-1$, the identity $u_i=u_{i+1}$ holds in one of the varieties $\mathcal X$ and $\mathcal Y$. Let $u_0,u_1,\dots,u_n$ be the shortest sequence of words with such properties. If all the words $u_0,u_1,\dots,u_n$ are not letters then $u_0=u_1=\cdots=u_n$ holds in $\mathcal{ZM}$. This means that the sequence of words $u_0,u_1,\dots,u_n$ is a deduction of the identity $u=v$ from identities of the varieties $\mathcal{ZM\vee X}$ and $\mathcal{ZM\vee Y}$, whence $u=v$ holds in $\mathcal{(ZM\vee X)\wedge(ZM\vee Y)}$. Let now $i$ be an index such that $u_i\equiv x$ for some letter $x$. Clearly, $0<i<n$ because the variety $\mathcal{ZM}$ satisfies the identity $u_0=u_n$ but does not satisfy any identity of the kind $x=w$. The identity $u_{i-1}=x$ holds in one of the varieties $\mathcal X$ and $\mathcal Y$, say in $\mathcal X$. Then $\mathcal Y$ satisfies the identity $x=u_{i+1}$. Since both the identities $u_{i-1}=x$ and $x=u_{i+1}$ fail in $\mathcal{ZM}$, we have that $\mathcal{ZM}$ is contained neither in $\mathcal X$ nor in $\mathcal Y$. Therefore, the varieties $\mathcal X$ and $\mathcal Y$ are completely regular. By Lemma~\ref{identity for cr} each of the identities $u_0=\,\overline{\overline{u_0}}$ and $\overline{\overline{u_n}}\,=u_n$ holds in one of the varieties $\mathcal X$ and $\mathcal Y$. Further, for each $i=0,1,\dots,n-1$ one of the varieties $\mathcal X$ and $\mathcal Y$ satisfies the identity $\overline{\overline{u_i}}\,=\,\overline{\overline{u_{i+1}}}$. The words $u_0$ and $u_n$ are not letters, whence the variety $\mathcal{ZM}$ satisfies the identities $u_0=\,\overline{\overline{u_0}}\,=\,\overline{\overline{u_1}}\,=\cdots=\,\overline{\overline{u_n}}\,=u_n$. Summarizing all we say, we obtain that the sequence of words $u_0,\overline{\overline{u_0}},\overline{\overline{u_1}},\dots,\overline{\overline{u_n}},u_n$ is a deduction of the identity $u=v$ from the identities of the varieties $\mathcal{ZM\vee X}$ and $\mathcal{ZM\vee Y}$. Therefore, this identity holds in $\mathcal{(ZM\vee X)\wedge(ZM\vee Y)}$.

\smallskip

\emph{Codistributivity.} In view of Lemma~\ref{SL and ZM are atoms}, if $\mathcal W$ is an arbitrary epigroup variety then either $\mathcal{W\supseteq ZM}$ or $\mathcal{W\wedge ZM=T}$. We need to verify that
$$\mathcal{ZM\wedge(X\vee Y)=(ZM\wedge X)\vee(ZM\wedge Y)}\ldotp$$
Clearly, both the parts of this equality equals $\mathcal{ZM}$ whenever at least one of the varieties $\mathcal X$ or $\mathcal Y$ contains $\mathcal{ZM}$. It remains to verify that if $\mathcal{X\nsupseteq ZM}$ and $\mathcal{Y\nsupseteq ZM}$ then $\mathcal{X\vee Y\nsupseteq ZM}$. This claim immediately follows from the well-known fact that an epigroup variety $\mathcal W$ does not contain the variety $\mathcal{ZM}$ if and only if $\mathcal W$ is completely regular.

\smallskip

\emph{Modularity.} For any epigroup variety $\mathcal X$, we put $\mathcal{\CR(X)=CR\wedge X}$ where $\mathcal{CR}$ is the variety of all completely regular epigroups. Suppose that $\mathcal{X\subseteq Y}$. We need to prove that
\begin{equation}
\label{inclusion for modularity}
\mathcal{(ZM\vee X)\wedge Y\subseteq(ZM\wedge Y)\vee X}
\end{equation}
because the opposite inclusion is evident. If $\mathcal{ZM\subseteq X}$ [respectively $\mathcal{ZM\subseteq Y}$] then both the parts of the inclusion~\eqref{inclusion for modularity} coincide with $\mathcal X$ [with $\mathcal{ZM\vee X}$]. Let now $\mathcal{ZM\nsubseteq X}$ and $\mathcal{ZM\nsubseteq Y}$. Then the varieties $\mathcal X$ and $\mathcal Y$ are completely regular. Therefore,
$$\mathcal{(ZM\vee X)\wedge Y\subseteq\CR(ZM\vee X)}\ldotp$$
Further, $\mathcal{ZM\wedge Y=T}$, whence the right part of the inclusion~\eqref{inclusion for modularity} coincides with $\mathcal X$. Let $u=v$ be an identity that holds in $\mathcal X$. Lemmas~\ref{identity for cr} and~\ref{identity for nil} imply that $\mathcal{ZM\vee X}$ satisfies the identity $\overline{\overline u}\,=\,\overline{\overline v}$. Therefore, $u=v$ in $\CR(\mathcal{ZM\vee X})$ by Lemma~\ref{identity for cr}. We have proved that $\mathcal{\CR(ZM\vee X)\subseteq X}$, whence
$$\mathcal{(ZM\vee X)\wedge Y\subseteq\CR(ZM\vee X)\subseteq X=T\vee X=(ZM\wedge Y)\vee X}\ldotp$$
Proposition is proved.
\end{proof}

For convenience of references, we formulate the following fact that immediately follows from Lemmas~\ref{join with neutral atom} and~\ref{SL and ZM are atoms} and Propositions~\ref{SL is neutral} and~\ref{ZM is neutral}.

\begin{corollary}
\label{join with SL or ZM or SL+ZM}
Let $I$ be a non-trivial lattice identity and $\mathcal W$ is one of the varieties $\mathcal{SL}$, $\mathcal{ZM}$ or $\mathcal{SL\vee ZM}$. An epigroup variety $\mathcal X$ is an $I$-element of the lattice $\mathbf{EPI}$ if and only if the variety $\mathcal{X\vee W}$ has the same property.\qed
\end{corollary}

\section{Upper-modular varieties}
\label{upper-modular proof}

Here we verify Theorem~\ref{upper-modular}. To do this, we need several auxiliary statements.

\begin{lemma}
\label{umod str permut is commut}
If a strongly permutative epigroup variety $\mathcal V$ is an upper-modular element of the lattice $\mathbf{EPI}$ then $\mathcal V$ is commutative.
\end{lemma}

\begin{proof}
In view of Corollary~\ref{str permut decomposition}, $\mathcal{V=\mathcal G\vee C}_m\vee\mathcal N$ where $\mathcal G$ is an Abelian group variety, $m\ge0$ and $\mathcal N$ is a nilvariety. If $\deg(\mathcal V)\le2$ then $\mathcal{N\subseteq ZM}$, and we are done. Let now $\deg(\mathcal V)>2$. By Proposition~\ref{degree} $\mathcal V$ contains the variety $\mathcal X=\var\{x^2=xyz=0,\,xy=yx\}$. Suppose that $\mathcal V$ is not commutative. Let $\mathcal G'$ be a non-abelian group variety. Since $\mathcal V$ is strongly permutative, every group in $\mathcal V$ is abelian. Therefore, the variety $\mathcal{(G'\wedge V)\vee X}$ is commutative. Since $\mathcal{X\subseteq V}$ and the variety $\mathcal V$ is upper-modular, we have that $\mathcal{(G'\wedge V)\vee X=(G'\vee X)\wedge V}$. We see that the variety $\mathcal{(G'\vee X)\wedge V}$ is commutative. Hence there is a deduction of the identity $xy=yx$ from the identities of the varieties $\mathcal{G'\vee X}$ and $\mathcal V$. In particular, there is a word $v$ such that $v\not\equiv xy$ and the identity $xy=v$ holds either in $\mathcal{G'\vee X}$ or in $\mathcal V$. The claims~(i) and~(iii) of Lemma~\ref{splitting} imply that a variety with the identity $xy=v$ is either commutative or a variety of degree $\le2$. The variety $\mathcal{G'\vee X}$ is neither commutative (because $\mathcal G'$ is non-abelian) nor a variety of degree $\le2$ (because $\deg(\mathcal X)>2$). Since $\deg(\mathcal V)>2$, we have that $\mathcal V$ is commutative.
\end{proof}

A semigroup variety is called \emph{proper} if it differs from the variety of all semigroups. It is proved in~\cite[Theorem~1.1]{Vernikov-08b} that if $\mathcal V$ is a proper upper-modular in \textbf{SEM} variety then, first, $\mathcal V$ is periodic\footnote{Note that this is a very special case of the following result obtained in~\cite[Theorem~1]{Shaprynskii-12b}: if $I$ is a non-trivial lattice identity then a proper semigroup variety is periodic whenever it is an $I$-element of the lattice \textbf{SEM}.}, and, second, every nilsubvariety of $\mathcal V$ is commutative and satisfies the identity~\eqref{xxy=xyy}. As we have already mentioned in Subsection~\ref{introduction epigroups}, the epigroup analog of the first claim is not true. Our next step is the following partial epigroup analog of the second claim.

\begin{proposition}
\label{umod nec}
If a strongly permutative epigroup variety $\mathcal V$ is an upper-modular element of the lattice $\mathbf{EPI}$ then every nil-semi\-group in $\mathcal V$ satisfies the identity~\eqref{xxy=xyy}.
\end{proposition}

\begin{proof}
According to Lemma~\ref{umod str permut is commut} the variety $\mathcal V$ is commutative. If all nil-semi\-groups in $\mathcal V$ are singleton then the desirable conclusion is evident. Suppose now that $\mathcal V$ contains a non-single\-ton nil-semi\-group $N$. Let $\mathcal N$ be the variety generated by $N$. Clearly, this variety is commutative. It is evident that $\mathcal{ZM\subseteq N}$. We need to verify that $\mathcal N$ satisfies the identity~\eqref{xxy=xyy}. Put
$$\mathcal I=\var\{x^2y=xy^2,\,xy=yx,\,x^2yz=0\}$$
and $\mathcal{N'=N\wedge I}$. It is clear that $\mathcal{ZM\subseteq I}$, whence $\mathcal{ZM\subseteq N'}$. 

A semigroup analog of the proposition we verify is proved in~\cite{Vernikov-08b} (see the last paragraph of Section~3 in that article). The arguments used there are based on the fact that there is a variety $\mathcal X$ such that the following two claims are valid:
\begin{itemize}
\item[(i)] $\mathcal{(X\vee N')\wedge V\subseteq I}$;
\item[(ii)] if $v\in\{x^2y,xyx,yx^2\}$ and $w\in\{xy^2,yxy,y^2x\}$ then the identity $v=w$ fails in $\mathcal X$.
\end{itemize}
In~\cite{Vernikov-08b} some periodic group variety plays the role of $\mathcal X$. Here we should take another $\mathcal X$. Namely, put $\mathcal{X=LZM\vee RZM}$ where
$$\mathcal{LZM}=\var\{xyz=xy\}\quad\text{and}\quad\mathcal{RZM}=\var\{xyz=yz\}\ldotp$$
The variety $\mathcal X$ satisfies the identity $xyzxy=xy$. Therefore, Lemma~\ref{word problem}(i) implies that $\mathcal{SL\nsubseteq X}$. Further, substituting~1 for $x$ and $y$ in the identity $xyzxy=xy$, we obtain that all groups in $\mathcal X$ are singleton. Hence every commutative semigroup in $\mathcal X$ is a nil-semi\-group. Further, $\mathcal X$ satisfies the identity $xy=(xy)^2$, whence all nil-semi\-groups in $\mathcal X$ lie in $\mathcal{ZM}$ by Lemma~\ref{splitting}(ii). Since the variety $\mathcal{X\wedge V}$ is commutative, $\mathcal{X\wedge V\subseteq ZM}$. The variety $\mathcal V$ is upper-modular and $\mathcal{N'\subseteq V}$. Therefore,
$$\mathcal{(X\vee N')\wedge V=(X\wedge V)\vee N'\subseteq ZM\vee N'=N'\subseteq I}\ldotp$$
We have proved the claim~(i). To verify the claim~(ii), we note that if $\mathcal X$ satisfies a semigroup identity $v=w$ then the words $v$ and $w$ have the same prefix of length~2 and the same suffix of length~2. Clearly, this is not the case whenever $v\in\{x^2y,xyx,yx^2\}$ and $w\in\{xy^2,yxy,y^2x\}$. Now we can complete the proof by the same arguments as in the last paragraph of~\cite[Section~3]{Vernikov-08b}.
\end{proof}

The proof of the following statement repeats almost literally the `only if' part of the proof of Theorem~2 in~\cite{Vernikov-Volkov-06}.

\begin{proposition}
\label{umod suf}
If a nilvariety of epigroups $\mathcal X$ satisfies the identities~\eqref{xxy=xyy} and $xy=yx$ then $\mathcal X$ is an upper-modular element of the lattice $\mathbf{EPI}$.
\end{proposition}

\begin{proof}
It is easy to prove (see~\cite[Lemma~2.7]{Vernikov-Volkov-06}, for instance) that $\mathcal X$ satisfies the identity $x^2yz=0$. Thus, $\mathcal{X\subseteq I}$. Put
$$U=\{x^2,x^3,x^2y,x_1x_2\cdots x_n\mid n\in\mathbb N\}\ldotp$$
It is evident that any subvariety of $\mathcal I$ may be given in $\mathcal I$ only by identities of the type $u=v$ or $u=0$ where $u,v\in U$. Lemma~\ref{splitting} implies that if $u,v\in U$ and $u\not\equiv v$ then $u=v$ implies in $\mathcal I$ the identity $u=0$. Now it is very easy to check that the lattice $L(\mathcal I)$ has the form shown on Fig.~\ref{L(I)} where
\begin{align*}
\mathcal I_n&=\var\{x^2yz=x_1x_2\cdots x_n=0,\,x^2y=xy^2,\,xy=yx\}\enskip\text{where }n\ge4,\\
\mathcal J&=\var\{x^2yz=x^3=0,\,x^2y=xy^2,\,xy=yx\},\\
\mathcal J_n&=\var\{x^2yz=x^3=x_1x_2\cdots x_n=0,\,x^2y=xy^2,\,xy=yx\}\enskip\text{where }n\ge4,\\
\mathcal K&=\var\{x^2y=0,\,xy=yx\},\\
\mathcal K_n&=\var\{x^2y=x_1x_2\cdots x_n=0,\,xy=yx\}\enskip\text{where }n\ge3,\\
\mathcal L&=\var\{x^2=0,\,xy=yx\},\\
\mathcal L_n&=\var\{x^2=x_1x_2\cdots x_n=0,\,xy=yx\}\enskip\text{where }n\in\mathbb N\ldotp
\end{align*}
Note that $\mathcal L_1=\mathcal T$ and $\mathcal L_2=\mathcal{ZM}$.

\begin{figure}[tbh]
\begin{center}
\unitlength=1mm
\linethickness{.4pt}
\begin{picture}(51,96)
\put(5,5){\line(0,1){32}}
\put(5,25){\line(3,2){15}}
\put(5,35){\line(3,2){45}}
\put(5,43){\line(0,1){14}}
\put(5,45){\line(3,2){45}}
\put(5,55){\line(3,2){45}}
\put(5,65){\line(3,2){45}}
\put(20,35){\line(0,1){12}}
\put(20,53){\line(0,1){14}}
\put(35,55){\line(0,1){2}}
\put(35,63){\line(0,1){14}}
\put(50,65){\line(0,1){2}}
\put(50,73){\line(0,1){14}}
\put(5,5){\circle*{1.33}}
\put(5,15){\circle*{1.33}}
\put(5,25){\circle*{1.33}}
\put(5,35){\circle*{1.33}}
\put(5,45){\circle*{1.33}}
\put(5,55){\circle*{1.33}}
\put(5,65){\circle*{1.33}}
\put(20,35){\circle*{1.33}}
\put(20,45){\circle*{1.33}}
\put(20,55){\circle*{1.33}}
\put(20,65){\circle*{1.33}}
\put(20,75){\circle*{1.33}}
\put(35,55){\circle*{1.33}}
\put(35,65){\circle*{1.33}}
\put(35,75){\circle*{1.33}}
\put(35,85){\circle*{1.33}}
\put(50,65){\circle*{1.33}}
\put(50,75){\circle*{1.33}}
\put(50,85){\circle*{1.33}}
\put(50,95){\circle*{1.33}}
\put(5,38.5){\circle*{.52}}
\put(5,40){\circle*{.52}}
\put(5,41.5){\circle*{.52}}
\put(5,58.5){\circle*{.52}}
\put(5,60){\circle*{.52}}
\put(5,61.5){\circle*{.52}}
\put(5,63.5){\circle*{.52}}
\put(20,48.5){\circle*{.52}}
\put(20,50){\circle*{.52}}
\put(20,51.5){\circle*{.52}}
\put(20,68.5){\circle*{.52}}
\put(20,70){\circle*{.52}}
\put(20,71.5){\circle*{.52}}
\put(20,73){\circle*{.52}}
\put(35,58.5){\circle*{.52}}
\put(35,60){\circle*{.52}}
\put(35,61.5){\circle*{.52}}
\put(35,78.5){\circle*{.52}}
\put(35,80){\circle*{.52}}
\put(35,81.5){\circle*{.52}}
\put(35,83){\circle*{.52}}
\put(50,68.5){\circle*{.52}}
\put(50,70){\circle*{.52}}
\put(50,71.5){\circle*{.52}}
\put(50,88.5){\circle*{.52}}
\put(50,90){\circle*{.52}}
\put(50,91.5){\circle*{.52}}
\put(50,93){\circle*{.52}}
\put(51,96){\makebox(0,0)[lc]{$\mathcal I$}}
\put(51,64){\makebox(0,0)[lc]{$\mathcal I_4$}}
\put(51,74){\makebox(0,0)[lc]{$\mathcal I_n$}}
\put(51,84){\makebox(0,0)[lc]{$\mathcal I_{n+1}$}}
\put(34,87){\makebox(0,0)[rc]{$\mathcal J$}}
\put(36,54){\makebox(0,0)[lc]{$\mathcal J_4$}}
\put(36,64){\makebox(0,0)[lc]{$\mathcal J_n$}}
\put(36,74){\makebox(0,0)[lc]{$\mathcal J_{n+1}$}}
\put(19,77){\makebox(0,0)[rc]{$\mathcal K$}}
\put(21,33){\makebox(0,0)[lc]{$\mathcal K_3$}}
\put(21,43){\makebox(0,0)[lc]{$\mathcal K_4$}}
\put(21,53){\makebox(0,0)[lc]{$\mathcal K_n$}}
\put(21,63){\makebox(0,0)[lc]{$\mathcal K_{n+1}$}}
\put(4,67){\makebox(0,0)[rc]{$\mathcal L$}}
\put(4,4){\makebox(0,0)[rc]{$\mathcal L_1$}}
\put(4,14){\makebox(0,0)[rc]{$\mathcal L_2$}}
\put(4,24){\makebox(0,0)[rc]{$\mathcal L_3$}}
\put(4,34){\makebox(0,0)[rc]{$\mathcal L_4$}}
\put(4,44){\makebox(0,0)[rc]{$\mathcal L_n$}}
\put(4,54){\makebox(0,0)[rc]{$\mathcal L_{n+1}$}}
\end{picture}
\caption{The lattice $L(\mathcal I)$}
\label{L(I)}
\end{center}
\end{figure}
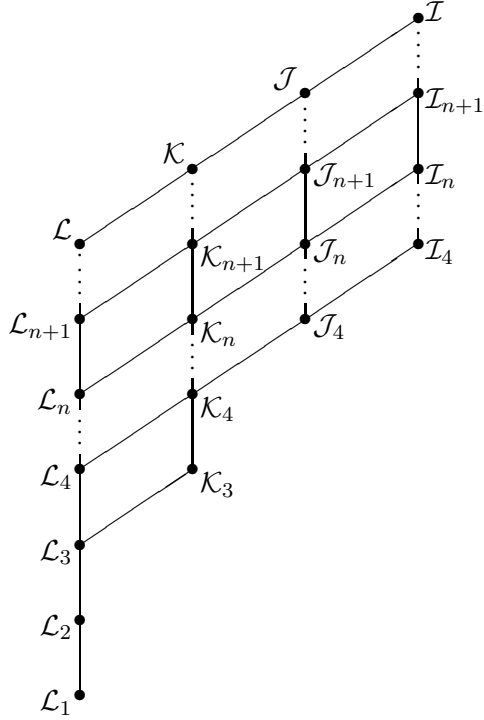

Let $\mathcal{X\subseteq I}$. We have to check that if $\mathcal{Y\subseteq X}$ and $\mathcal Z$ is an arbitrary epigroup variety then $\mathcal{(Z\vee Y)\wedge X=(Z\wedge X)\vee Y}$. For a variety $\mathcal M$ with $\mathcal{M\subseteq I}$, we denote by $\mathcal M^\ast$ the least of the varieties $\mathcal I$, $\mathcal J$, $\mathcal K$ and $\mathcal L$ that contains $\mathcal M$. Fig.~\ref{L(I)} shows that if $\mathcal M_1,\mathcal M_2\subseteq\mathcal I$ then $\mathcal M_1=\mathcal M_2$ if and only if $\deg(\mathcal M_1)=\deg(\mathcal M_2)$ and $\mathcal M_1^\ast=\mathcal M_2^\ast$. Therefore, we have to verify the following two equalities:
\begin{align}
\label{the same degree}&\mathcal{\deg\bigl((Z\vee Y)\wedge X\bigr)=\deg\bigl((Z\wedge X)\vee Y\bigr)},\\
\label{the same star}&\mathcal{\bigl((Z\vee Y)\wedge X\bigr)^\ast=\bigl((Z\wedge X)\vee Y\bigr)^\ast}\ldotp
\end{align}

\smallskip

\emph{The equality}~\eqref{the same degree}. Put $\deg(\mathcal X)=k$, $\deg(\mathcal Y)=\ell$ and $\deg(\mathcal Z)=m$. Clearly, $\ell\le k$ because $\mathcal{Y\subseteq X}$. The set of all natural numbers with the operations $\min$ and $\max$ is a distributive lattice. Now Corollaries~\ref{degree of meet} and~\ref{degree of join with nil} apply and we conclude that
\begin{align*}
\mathcal{\deg\bigl((Z\vee Y)\wedge X\bigr)}&=\min\bigl\{\max\{m,\ell\},k\bigr\}=\max\bigl\{\min\{m,k\},\min\{\ell,k\}\bigr\}\\
&=\max\bigl\{\min\{m,k\},\ell\bigr\}=\mathcal{\deg\bigl((Z\wedge X)\vee Y\bigr)}\ldotp
\end{align*}
The equality~\eqref{the same degree} is proved.

\smallskip

\emph{The equality}~\eqref{the same star}. Clearly, this equality is equivalent to the following claim: if $u$ is one of the words $x^3$, $x^2y$ and $x^2$ then the variety $\mathcal{(Z\vee Y)\wedge X}$ satisfies the identity $u=0$ if and only if the variety $\mathcal{(Z\wedge X)\vee Y}$ does so. It suffices to verify that $u=0$ holds in $\mathcal{(Z\vee Y)\wedge X}$ whenever it is so in $\mathcal{(Z\wedge X)\vee Y}$ because the opposite claim immediately follows from the evident inclusion $\mathcal{(Z\wedge X)\vee Y\subseteq(Z\vee Y)\wedge X}$. Further considerations are divided into two cases.

\smallskip

\emph{Case}~1: $u\equiv x^n$ where $n\in\{2,3\}$. Then $x^n=0$ in $\mathcal{(Z\wedge X)\vee Y}$. This means that $x^n=0$ in $\mathcal Y$ and there is a deduction of the identity $x^n=0$ from the identities of the varieties $\mathcal Z$ and $\mathcal X$. In particular, there is a word $v$ such that $v\not\equiv x^n$ and $x^n=v$ holds in either $\mathcal Z$ or $\mathcal X$. If $x^n=v$ in $\mathcal X$ then the fact that $\mathcal X$ is a nil-variety together with the claims~(i) and~(ii) of Lemma~\ref{splitting} imply that $x^n=0$ in $\mathcal X$, and moreover in $\mathcal{(Z\vee Y)\wedge X}$. Let now $x^n=v$ in $\mathcal Z$. The case when $v$ is a semigroup word may be considered by the same arguments as in the Case~1 in the `only if' part of the proof of Theorem~2 in~\cite{Vernikov-Volkov-06}. Let now $v$ is a non-semi\-group word. Then, in view of Lemma~\ref{identity for nil}, the identity $v=0$ holds in the variety $\mathcal X$, and therefore in $\mathcal Y$. Recall that $x^n=0$ in $\mathcal Y$. Therefore, the identity $x^n=v$ holds in $\mathcal Y$. Since this identity holds in $\mathcal Z$ as well, we obtain that it holds in $\mathcal{Y\vee Z}$. We see that the sequence $x^n$, $v$, 0 is a deduction of the identity $x^n=0$ from the identities of the varieties $\mathcal{Y\vee Z}$ and $\mathcal X$, whence $x^n=0$ holds in $\mathcal{(Y\vee Z)\wedge X}$.

\smallskip

\emph{Case}~2: $u\equiv x^2y$. We have to check that if the variety $\mathcal{(Z\wedge X)\vee Y}$ satisfies the identity~\eqref{xxy=0} then the variety $\mathcal{(Z\vee Y)\wedge X}$ also satisfies this identity. Put
$$W=\{x^2y,xyx,yx^2,y^2x,yxy,xy^2\}\ldotp$$
The variety $\mathcal{(Z\vee Y)\wedge X}$ is commutative. Therefore, it suffices to verify that this variety satisfies an identity $w=0$ for some word $w\in W$. By the hypothesis the variety $\mathcal{(Z\wedge X)\vee Y}$ satisfies the identity~\eqref{xxy=0}. This means that this identity holds in $\mathcal Y$ and there is a deduction of the identity from identities of the varieties $\mathcal X$ and $\mathcal Z$. Let $x^2y\equiv u_0,u_1,\dots,u_n,0$ be an arbitrary such deduction. The case when $u_n\in W$ may be considered by the same way as in the `only if' part of the proof of Theorem~2 in~\cite{Vernikov-Volkov-06}.

Let now $u_n\notin W$. Since $u_0\in W$, there is an index $i>0$ such that $u_i\notin W$ while $u_{i-1}\in W$. The identity $u_{i-1}=u_i$ holds in one of the varieties $\mathcal Z$ and $\mathcal X$. If $u_{i-1}=u_i$ holds in $\mathcal X$ then $\mathcal X$ satisfies the identity $u_{i-1}=0$ (this follows from~\cite[Lemma~2.5]{Vernikov-Volkov-06} whenever $u_i$ is a semigroup word and from Lemma~\ref{identity for nil} otherwise). Therefore, $u_{i-1}=0$ holds in $\mathcal{(Z\vee Y)\wedge X}$. Since $u_{i-1}\in W$, we are done.

Finally, suppose that $u_{i-1}=u_i$ holds in $\mathcal Z$. If $u_i$ is a semigroup word then we may complete the proof by the same arguments as in the `only if' part of the proof of Theorem~2 in~\cite{Vernikov-Volkov-06}. Suppose now that $u_i$ is not a semigroup word. Lemma~\ref{identity for nil} implies then that the variety $\mathcal Y$ satisfies the identity $u_i=0$. Hence the identity $u_{i-1}=u_i$ holds in $\mathcal Y$, and therefore the variety $\mathcal{Y\vee Z}$ satisfies this identity. Applying Lemma~\ref{identity for nil} again, we conclude that $u_i=0$ holds in $\mathcal X$. Whence, the sequence $u_{i-1}$, $u_i$, 0 is a deduction of the identity $u_{i-1}=0$ from the identities of the varieties $\mathcal{Y\vee Z}$ and $\mathcal X$. Thus $u_{i-1}=0$ holds in $\mathcal{(Y\vee Z)\wedge X}$, and we are done.

The equality~\eqref{the same star} is proved. Thus, we have proved Proposition~\ref{umod suf}.
\end{proof}

\emph{Proof of Theorem}~\ref{upper-modular}. \emph{Necessity.} Let $\mathcal V$ be a strongly permutative upper-modular epigroup variety. In view of Corollary~\ref{str permut decomposition}, $\mathcal{V=\mathcal G\vee C}_m\vee\mathcal N$ where $\mathcal G$ is an abelian group variety, $m\ge0$ and $\mathcal N$ is a nilvariety. Lemma~\ref{umod str permut is commut} and Proposition~\ref{umod nec} imply respectively that $\mathcal N$ is commutative and satisfies the identity~\eqref{xxy=xyy}. The variety $\mathcal C_m$ contains a nilsubvariety $\var\{x^m=0,\,xy=yx\}$. Clearly, this variety does not satisfy the identity~\eqref{xxy=xyy} whenever $m\ge3$. Now Proposition~\ref{umod nec} applies again and we conclude that $m\le2$. If the variety $\mathcal N$ satisfies the identity~\eqref{xxy=0} then the claim~(ii) of Theorem~\ref{upper-modular} holds. Suppose now that the identity~\eqref{xxy=0} fails in $\mathcal N$. By~\cite[Lemma~7]{Volkov-89} this implies that $\mathcal N$ contains the variety $\mathcal J$. We need to verify that $\mathcal{G=T}$ and $m\le1$. Arguing by contradiction, suppose that either $\mathcal{G\ne T}$ or $m\ge2$. Then $\mathcal V$ contains a variety of the form $\mathcal{X\vee J}$ where $\mathcal X$ is either a non-trivial abelian group variety or the variety $\mathcal C_2$. It is well known and may be easily checked that the variety $\mathcal C_2$ is generated by the epigroup 
$$C_2=\langle e,a\mid e^2=e,\,ea=ae=a,\,a^2=0\rangle=\{e,a,0\}$$
and $e$ is the unit of $C_2$. Thus, $\mathcal M$ is generated by an epigroup with unit in any case. Suppose that $\mathcal X$ satisfies the identity~\eqref{xxy=xyy}. Substituting~1 for $y$ in this identity, we have that $x^2=x$ holds in $\mathcal X$. But this identity is false both in a non-trivial group variety and in the variety $\mathcal C_2$. As it is verified in the proof of~\cite[Lemma~8]{Volkov-89}, this implies that~\eqref{xxy=xyy} is false in any nil-semi\-group in $\mathcal{X\vee J}$. But this contradicts Proposition~\ref{umod nec}.

\smallskip

\emph{Sufficiency.} If $\mathcal V$ satisfies the claim~(i) of Theorem~\ref{upper-modular} then $\mathcal V$ is upper-modular by Proposition~\ref{umod suf} and Corollary~\ref{join with SL or ZM or SL+ZM}. Suppose now that $\mathcal V$ satisfies the claim~(ii). In other words, $\mathcal{V=G\vee C}_m\vee\mathcal N$ where $\mathcal G$ is an abelian group variety, $0\le m\le2$, and $\mathcal N$ satisfies the identities $xy=yx$ and~\eqref{xxy=0}. Let $\mathcal{Y\subseteq V}$ and $\mathcal Z$ an arbitrary epigroup variety. We aim to verify that
$$\mathcal{(Z\vee Y)\wedge V=(Z\wedge V)\vee Y}\ldotp$$

As we have already mentioned in the proof of Lemma~\ref{monoid decomposition}, the variety $\mathcal C_m$ is generated by the $(m+1)$-element combinatorial cyclic monoid $C_m$ and the set $X=\{m\in\mathbb N\mid C_m\in\mathcal X\}$ has the greatest element. For any $m\ge0$, let $c_m$ be a generator of $C_m$. Put $C=\mathop\prod\limits_{m\in X}C_m$. Then the semigroup $C$ is not an epigroup because no power of the element $(\dots,c_m,\dots)_{m\in X}$ belongs to a subgroup of $C$. Thus, the set $X$ has the greatest element. We denote this element by $m$ and put $C\mathcal{(X)=C}_m$. It is clear that the varieties $\mathcal{(Z\vee Y)\wedge V}$ and $\mathcal{(Z\wedge V)\vee Y}$ are commutative. In view of Corollary~\ref{str permut decomposition}, it suffices to verify the following three equalities:
\begin{align}
\label{the same Gr-part}
\Gr\bigl(\mathcal{(Z\vee Y)\wedge V}\bigr)&=\Gr\bigl(\mathcal{(Z\wedge V)\vee Y}\bigr),\\
\label{the same C-part}
C\bigl(\mathcal{(Z\vee Y)\wedge V}\bigr)&=C\bigl(\mathcal{(Z\wedge V)\vee Y}\bigr),\\
\label{the same Nil-part}
\Nil\bigl(\mathcal{(Z\vee Y)\wedge V}\bigr)&=\Nil\bigl(\mathcal{(Z\wedge V)\vee Y}\bigr)\ldotp
\end{align}

\smallskip

\emph{The equality}~\eqref{the same Gr-part}. If $\mathcal G$ is a periodic group variety then we denote by $\exp(G)$ the \emph{exponent} of $\mathcal G$, that is the least number $n$ such that $\mathcal G$ satisfies the identity $x=x^{n+1}$. For a non-periodic group variety $\mathcal G$, we put $\exp(\mathcal G)=\infty$. Put $\mathcal G_1=\Gr\bigl(\mathcal{(Z\vee Y)\wedge V}\bigr)$ and $\mathcal G_2=\Gr\bigl(\mathcal{(Z\wedge V)\vee Y}\bigr)$. Since $\mathcal G_1,\mathcal G_2\subseteq\mathcal V$, we have that $\mathcal G_1$ and $\mathcal G_2$ are abelian group varieties. To prove the equality~\eqref{the same Gr-part}, it suffices to verify that $\exp(\mathcal G_1)=\exp(\mathcal G_2)$. This claim is verified by the same arguments as the analogous claim in the proof of Theorem~1.2 in~\cite{Vernikov-08b}.

\smallskip

\emph{The equality}~\eqref{the same C-part}. Here and below we need the following easy remark. It is evident that if $m\ge3$ then the identity~\eqref{xxy=0} fails in the variety $\Nil(\mathcal C_m)=\var\{x^m=0,\,xy=yx\}$. This means that each part of the equality~\eqref{the same C-part} coincides with the variety $\mathcal C_m$ for some $0\le m\le2$. Then we may complete the proof of equality~\eqref{the same C-part} by the same arguments as in the proof of the equality~(4.2) in~\cite{Vernikov-08b}.

\smallskip

\emph{The equality}~\eqref{the same Nil-part}. The varieties $\mathcal G$, $\mathcal C_2$ and $\mathcal N$ satisfy the identity $x^2y=\,\overline{\overline x}\,^2y$. By Lemma~\ref{identity for nil} the variety $\Nil(\mathcal V)$ satisfies the identity~\eqref{xxy=0}. Fig.~\ref{L(I)} shows that it suffices to check the following two claims: first, the varieties $\mathcal{(Z\vee Y)\wedge V}$ and $\mathcal{(Z\wedge V)\vee Y}$ have the same degree, and second, the variety $\Nil\bigl(\mathcal{(Z\vee Y)\wedge V}\bigr)$ satisfies the identity
\begin{equation}
\label{xx=0}
x^2=0
\end{equation}
if and only if the variety $\Nil\bigl(\mathcal{(Z\wedge V)\vee Y)}\bigr)$ satisfies this identity. The former claim may be verified by the same way as in the proof of the equality~(4.3) in~\cite{Vernikov-08b} with references to Proposition~\ref{degree}, Corollary~\ref{degree of join with nil} and Corollary~\ref{degree of join with cr} of the present article rather than Proposition~2.11, Lemma~2.13 and Lemma~2.12 of~\cite{Vernikov-08b} respectively. 

It remains to verify that the variety $\Nil\bigl(\mathcal{(Z\vee Y)\wedge V}\bigr)$ satisfies the identity~\eqref{xx=0} whenever $\Nil\bigl(\mathcal{(Z\wedge V)\vee Y}\bigr)$ does so (because the opposite claim is evident). Suppose that the identity~\eqref{xx=0} holds in $\Nil\bigl(\mathcal{(Z\wedge V)\vee Y}\bigr)$. Corollary~\ref{str permut decomposition} imply that the variety $\mathcal{(Z\wedge V)\vee Y}$ is the join of some group variety, the variety $\mathcal C_m$ for some $m\ge0$ and the variety $\Nil\bigl(\mathcal{(Z\wedge V)\vee Y}\bigr)$. Here $m\le2$ because $\mathcal{(Z\wedge V)\vee Y\subseteq V}$. This implies that the variety $\mathcal{(Z\wedge V)\vee Y}$ satisfies the identity $x^2=\,\overline{\overline{x^2}}$. In particular, this identity holds in both the varieties $\mathcal Y$ and $\mathcal{Z\wedge V}$. Therefore, there is a sequence of words $u_0,u_1,\dots,u_k$ such that $u_0\equiv x^2$, $u_k\equiv\,\overline{\overline{x^2}}$ and, for each $i=0,1,\dots,k-1$, the identity $u_i=u_{i+1}$ holds in one of the varieties $\mathcal Z$ or $\mathcal V$. We may assume that $u_i\not\equiv u_{i+1}$ for each $i=0,1,\dots,k-1$. Arguments from the proof of the equality~(4.3) in~\cite{Vernikov-08b} show that it suffices to check that the identity~\eqref{xx=0} holds in one of the varieties $\Nil(\mathcal Z)$ or $\Nil(\mathcal V)$. This fact follows from Lemma~\ref{splitting} whenever $u_1$ is a semigroup word and from Lemma~\ref{identity for nil} otherwise.

We complete the proof of Theorem~\ref{upper-modular}.\qed

\medskip

\begin{corollary}
\label{upper-modular str permut in SEM and EPI}
A periodic strongly permutative semigroup variety $\mathcal V$ is an upper-modular element of the lattice $\mathbf{SEM}$ if and only if $\mathcal V$ is an upper-modular element of the lattice $\mathbf{EPI}$.
\end{corollary}

\begin{proof}
\emph{Necessity.} Let $\mathcal V$ be a periodic strongly permutative semigroup variety and $\mathcal V$ is an upper-modular element of \textbf{SEM}. It follows from results of~\cite{Head-68} and the proof of Proposition~1 in~\cite{Volkov-89} that $\mathcal{V=G\vee C}_m\vee\mathcal N$ for some Abelian periodic group variety $\mathcal V$, some $m\ge0$ and some nilvariety $\mathcal N$. By~\cite[Theorem~1.1]{Vernikov-08b}, if $\mathcal X$ is a proper semigroup variety that is an upper-modular element of \textbf{SEM} and $\mathcal Y$ is a nilsubvariety of $\mathcal X$ then $\mathcal Y$ is commutative. Thus $\mathcal N$ is commutative. This implies that the variety $\mathcal V$ is commutative too. Now we may apply~\cite[Theorem~1.2]{Vernikov-08b} and conclude that $\mathcal V$ satisfies one of the claims~(i) or~(ii) of Theorem~\ref{upper-modular}. Therefore $\mathcal V$ is an upper-modular variety.

\smallskip

\emph{Sufficiency.} Let $\mathcal V$ be a periodic strongly permutative semigroup variety and $\mathcal V$ is an upper-modular variety. By Theorem~\ref{upper-modular}, either $\mathcal V$ satisfies the claim~(i) of this theorem or $\mathcal V$ satisfies the claim~(ii) of this theorem and the variety $\mathcal G$ from this claim is periodic. In both these cases $\mathcal V$ is an upper-modular element of \textbf{SEM} by~\cite[Theorem~1.2]{Vernikov-08b}.
\end{proof}

Theorem~\ref{upper-modular} readily implies the following

\begin{corollary}
\label{umod has distr lat}
If a strongly permutative epigroup variety $\mathcal V$ is an upper-modular element of the lattice $\mathbf{EPI}$ then the lattice $L(\mathcal V)$ is distributive.{\sloppy

}
\end{corollary}

\begin{proof}
If $\mathcal V$ satisfies the claim~(i) of Theorem~\ref{upper-modular} then it is periodic, whence it may be considered as a semigroup variety. In this case, it suffices to take into account a description of commutative semigroup varieties with distributive subvariety lattice obtained in~\cite{Volkov-91}. Suppose now that $\mathcal V$ satisfies the claim~(ii) of Theorem~\ref{upper-modular}. Then $\mathcal{V\subseteq AG\vee C}_2\vee\mathcal N$ where $\mathcal N$ satisfies the commutative law and the identity~\eqref{xxy=0}. In view of Proposition~\ref{direct product}, $L(\mathcal V)$ is embeddable into the direct product of the lattices $L(\mathcal{AG})$ and $L(\mathcal C_2\vee\mathcal N)$. The former lattice is generally known to be distributive. Finally, the variety $\mathcal C_2\vee\mathcal N$ is periodic, whence it may be considered as a semigroup variety. To complete the proof, it remains to note that the lattice $L(\mathcal C_2\vee\mathcal N)$ is distributive by the mentioned result of~\cite{Volkov-91}.
\end{proof}

\section{Codistributive varieties}
\label{codistributive proof}

Here we verify Theorem~\ref{codistributive}.

\smallskip

\emph{Necessity.} Here and in Section~\ref{neutral and costandard proof} we need the following statement that may be verified by repeating literally arguments from the first paragraph of the proof of Theorem~1.1 in~\cite{Vernikov-11}.

\begin{lemma}
\label{codistr without P and P*}
Let an epigroup variety $\mathcal V$ be a codistributive element of the lattice $\mathbf{EPI}$. If $\mathcal V$ does not contain the varieties $\mathcal P$ and $\overleftarrow{\mathcal P}$ then $\mathcal V$ is a variety of degree $\le2$.\qed
\end{lemma}

Let now $\mathcal V$ be a strongly permutative codistributive variety. It is evident that $\mathcal V$ is upper-modular. Theorem~\ref{upper-modular} implies that the variety $\mathcal V$ is commutative. Hence $\mathcal P,\overleftarrow{\mathcal P}\nsubseteq\mathcal V$. Lemma~\ref{codistr without P and P*} implies that $\mathcal V$ is a variety of degree $\le2$. It remains to refer to Corollary~\ref{str permut decomposition} and the observation that the variety $\mathcal C_m$ has a degree $>2$ whenever $m\ge2$.

\smallskip

\emph{Sufficiency.} In view of Corollary~\ref{join with SL or ZM or SL+ZM}, it suffices to verify that an abelian group variety $\mathcal G$ is codistributive. Let $\mathcal Y$ and $\mathcal Z$ be arbitrary epigroup varieties. Every epigroup variety is either periodic or contains the variety $\mathcal{AG}$. Suppose that at least one of the varieties $\mathcal Y$ and $\mathcal Z$, say $\mathcal Y$, contains $\mathcal{AG}$. Then $\mathcal{Y\vee Z\supseteq Y\supseteq AG\supseteq G}$ and therefore,
$$\mathcal{G\wedge(Y\vee Z)=G=G\vee(G\wedge Z)=(G\wedge Y)\vee(G\wedge Z)}\ldotp$$
Hence we may assume that the varieties $\mathcal Y$ and $\mathcal Z$ are periodic. Now we may complete the proof by the same arguments as in the proof of the implication c)\,$\longrightarrow$\,a) of Theorem~1.2 in~\cite{Vernikov-11}.\qed

\medskip

Comparing Theorem~1.2 in~\cite{Vernikov-11} with Theorem~\ref{codistributive}, we obtain the following

\begin{corollary}
\label{codistributive str permut in SEM and EPI}
A periodic strongly permutative semigroup variety $\mathcal V$ is a codistributive element of the lattice $\mathbf{SEM}$ if and only if $\mathcal V$ is a codistributive element of the lattice $\mathbf{EPI}$.\qed
\end{corollary}

\section{Modular varieties}
\label{modular proof}

Here we are going to prove Theorems~\ref{modular nec}--\ref{modular commut}. We need several auxiliary facts.

\begin{proposition}
\label{modular is periodic}
If an epigroup variety $\mathcal V$ is a modular element of the lattice $\mathbf{EPI}$ then $\mathcal V$ is periodic.
\end{proposition}

\begin{proof}
Let $\mathcal V$ be a modular epigroup variety. Suppose that $\mathcal V$ is non-periodic. Being an epigroup variety, $\mathcal V$ satisfies the identity $x^n=x^nx^\omega$ for some natural $n$. Consider varieties
$$\mathcal N_1=\var\{x_1x_2\dots x_{n+3}=0\}\text{ and }\mathcal N_2=\var\{x_1x_2\dots x_{n+3}=0,\,x^{n+1}y=x^ny^2\}\ldotp$$
To prove that $\mathcal V$ is non-modular, we are going to check that the varieties $\mathcal V$, $\mathcal N_1$, $\mathcal N_2$, $\mathcal{V\vee N}_1$ and $\mathcal{V\wedge N}_1$ form the 5-element non-modular sublattice $N_5$. Note that $\mathcal N_2\subseteq\mathcal N_1$. Whence, to achieve our aim, it suffices to verify the equalities $\mathcal{V\vee N}_1=\mathcal{V\vee N}_2$ and $\mathcal{V\wedge N}_1=\mathcal{V\wedge N}_2$.

The inclusion $\mathcal{V\vee N}_2\subseteq\mathcal{V\vee N}_1$ is evident. It is evident also that a non-trivial identity $u=v$ holds in $\mathcal N_2$ if and only if either $\ell(u),\ell(v)\ge n+3$ or $u=v$ coincides with the identity $x^{n+1}y=x^ny^2$. Let $u=v$ be a non-trivial identity that is satisfied by the variety $\mathcal{V\vee N}_2$. Substituting $y^2$ for $y$ in the identity $x^{n+1}y=x^ny^2$, we have $x^{n+1}y^2=x^ny^4$ that implies $x^{n+3}=x^{n+4}$. Therefore, a variety satisfying $x^{n+1}y=x^ny^2$ is periodic. Since the identity $u=v$ holds in a non-periodic variety $\mathcal V$, it differs from the identity $x^{n+1}y=x^ny^2$. Therefore, $\ell(u),\ell(v)\ge n+3$. This implies that $u=v$ holds in $\mathcal N_1$ and therefore, in $\mathcal{V\vee N}_1$. Thus, $\mathcal{V\vee N}_1\subseteq\mathcal{V\vee N}_2$. The equality $\mathcal{V\vee N}_1=\mathcal{V\vee N}_2$ is proved.

The inclusion $\mathcal{V\wedge N}_2\subseteq\mathcal{V\wedge N}_1$ is evident. The variety $\mathcal{V\wedge N}_1$ is a nilvariety and is contained in $\mathcal V$. Since $\mathcal V$ satisfies $x^n=x^nx^\omega$, Lemma~\ref{splitting}(ii) implies that $x^n=0$ holds in $\mathcal{V\wedge N}_1$. Therefore, $x^{n+1}y=0=x^ny^2$ in $\mathcal{V\wedge N}_1$. We see that $\mathcal{V\wedge N}_1\subseteq\mathcal N_2$, whence $\mathcal{V\wedge N}_1\subseteq\mathcal{V\wedge N}_2$. The equality $\mathcal{V\wedge N}_1=\mathcal{V\wedge N}_2$ is proved as well.
\end{proof}

We denote by \textbf{PER} the lattice of all periodic semigroup varieties.

\begin{lemma}
\label{non-substitutive}
Let $\mathcal V$ be a nilvariety that is a modular element of the lattice $\mathbf{EPI}$. If $\mathcal V$ satisfies a non-substitutive identity $u=v$ then it satisfies also the identity $u=0$.
\end{lemma}

\begin{proof}
If the identity $u=v$ is not a semigroup one then Lemma~\ref{identity for nil} is applied with the desirable conclusion. So, we may assume that $u=v$ is a semigroup identity. Note that the variety $\mathcal V$ is periodic, whence it may be considered as a semigroup variety. Clearly, $\mathcal V$ is a modular element of the lattice \textbf{PER}. It is verified in~\cite[Proposition~2.2]{Vernikov-07a} that if a semigroup variety is modular in the lattice \textbf{SEM} then it has the property we verify. All varieties that appear in the proof of this claim are periodic. Therefore, the desirable conclusion is true for modular elements of the lattice \textbf{PER}, and we are done.
\end{proof}

The formulation of the following statement and its proof are closely related with the formulation and proof of Lemma~3.1 of the article~\cite{Shaprynskii-12a}. But we need slightly modify some terminology from this article. Lemma~3.1 of~\cite{Shaprynskii-12a} deals with the notions of equivalent and non-stable pairs of (semigroup) words defined in~\cite{Shaprynskii-12a}. Here we need some modification of the first notion and do not require the second one at all. So, we call semigroup words $u$ and $v$ \emph{equivalent} if $u\equiv\xi(v)$ for some automorphism $\xi$ on $F$. Clearly, if words $u$ and $v$ are equivalent semigroup words and $c(u)=c(v)$ then $u=v$ is a substitutive identity. 

\begin{lemma}
\label{identities for modular}
Let $\mathcal V$ be an epigroup variety that is a modular element of the lattice $\mathbf{EPI}$ and let $u$, $v$, $s$ and $t$ be pairwise non-equivalent words of the same length depending on the same letters. If the variety $\mathcal V$ satisfies the identities $u=v$ and $s=t$ then it satisfies also the identity $u=s$.
\end{lemma}

\begin{proof}
In view of Proposition~\ref{modular is periodic}, the variety $\mathcal V$ is periodic. Whence, it may be considered as a semigroup variety. Clearly, $\mathcal V$ is a modular element of the lattice \textbf{PER}. The proof of~\cite[Lemma~3.1]{Shaprynskii-12a} readily implies that if $u$, $v$, $s$ and $t$ are pairwise non-equivalent words of the same length depending on the same letters, $\mathcal V$ satisfies the identities $u=v$ and $s=t$ and $\mathcal V$ does not satisfy the identity $u=s$ then there are periodic varieties (in actual fact, even nilvarieties) $\mathcal U$ and $\mathcal W$ such that $\mathcal{U\subseteq W}$ but $\mathcal{(V\vee U)\wedge W\ne(V\wedge W)\vee U}$. This contradicts the claim that $\mathcal V$ is a modular element of the lattice \textbf{PER}.
\end{proof}

\begin{proof}[Proof of Theorem~\emph{\ref{modular nec}}]
Let $\mathcal V$ be a modular epigroup variety. According to Proposition~\ref{modular is periodic}, the variety $\mathcal V$ is periodic. It follows immediately from~\cite[Proposition~3.3]{Shaprynskii-12a} that if a periodic semigroup variety is a modular element of the lattice \textbf{SEM} then it is the join of one of the varieties $\mathcal T$ or $\mathcal {SL}$ and a nilvariety. Repeating literally arguments from the proof of this statement with references to Proposition~\ref{decomposition} and Lemma~\ref{monoid decomposition} of the present work rather than Lemma~2.6 of the article~\cite{Shaprynskii-12a} and to Lemma~\ref{identities for modular} of the present work rather than Lemma~3.1 of the article~\cite{Shaprynskii-12a}, we obtain that the variety $\mathcal V$ has the same property. Thus, $\mathcal{V=M\vee N}$ where $\mathcal M$ is one of the varieties $\mathcal T$ or $\mathcal {SL}$ and $\mathcal N$ is a nilvariety. It remains to verify that if $\mathcal N$ satisfies a non-substitutive identity $u=v$ then $\mathcal N$ satisfies also the identity $u=0$. If the identity $u=v$ is not a semigroup one then Lemma~\ref{identity for nil} is applied with the conclusion that $\mathcal N$ satisfies the identity $u=0$. So, we may assume that $u=v$ is a semigroup identity. Note that the variety $\mathcal N$ is periodic, whence it may be considered as a semigroup variety. In this situation the desirable conclusion directly follows from~\cite[Proposition~2.2]{Vernikov-07a}. Theorem~\ref{modular nec} is proved.
\end{proof}

Let \textbf{USEM} denotes the lattice of all varieties of unary semigroups. The following lemma will be helpful here and in Section~\ref{lower-modular proof}.

\begin{proposition}
\label{modular and lower-modular suf in USEM}
An arbitrary $0$-reduced epigroup variety is a modular and lower-modular element of the lattice $\mathbf{USEM}$.
\end{proposition}

\begin{proof}
Let $\mathcal N$ be a 0-reduced epigroup variety, while $\nu$ a fully invariant congruence on $F$ corresponding to $\mathcal N$. Then $\nu$ has exactly one non-singleton class (this class includes words that are equal to~0 in $\mathcal N$ and only them). The lattice \textbf{USEM} is antiisomorphic to the lattice of all fully invariant congruences on $F$. Further, the latter lattice is embeddable in the lattice $\Eq(F)$ of all equivalence relations on $F$. If an equivalence relation $\rho$ on a set $X$ has exactly one non-singleton class then $\rho$ is both modular and upper-modular element in the lattice $\Eq(X)$ (this claim is verified in~\cite[Proposition~2.2]{Jezek-81} for modular elements and in~\cite[Proposition~3]{Vernikov-Volkov-88} for upper-modular ones). Thus, $\nu$ is a modular and upper-modular element in the lattice $\Eq(F)$, and moreover in the lattice of all fully invariant congruences on $F$. Since the notion of a modular element of a lattice is self-dual, we are done.
\end{proof}

The lattice \textbf{EPI} is a sublattice in \textbf{USEM}. Whence, Proposition~\ref{modular and lower-modular suf in USEM} immediately implies Theorem~\ref{modular suf}.\qed

\medskip

\begin{proof}[Proof of Theorem~\emph{\ref{modular commut}}]
\emph{Necessity.} Let $\mathcal V$ be a commutative modular epigroup variety. By Theorem~\ref{modular nec} $\mathcal{V=M\vee N}$ where $\mathcal M$ is one of the varieties $\mathcal T$ or $\mathcal{SL}$ and $\mathcal N$ is a nilvariety. Corollary~\ref{join with SL or ZM or SL+ZM} implies that the variety $\mathcal N$ is modular. Since every commutative variety satisfies the identity $x^2y=yx^2$, Lemma~\ref{non-substitutive} implies that the identity~\eqref{xxy=0} holds in $\mathcal N$.

\smallskip

\emph{Sufficiency.} In view of Corollary~\ref{join with SL or ZM or SL+ZM}, it suffices to verify that a commutative epigroup variety satisfying the identity~\eqref{xxy=0} is modular. This fact may be verified by the same arguments as in the proof of the `if' part of Theorem~1 in~\cite{Vernikov-Volkov-06}. Theorem~\ref{modular commut} is proved.
\end{proof}

Theorems~\ref{modular commut} and~\ref{upper-modular} evidently imply

\begin{corollary}
\label{modular implies upper-modular}
If a commutative epigroup variety is a modular element of the lattice $\mathbf{EPI}$ then it is an upper-modular element of this lattice.\qed
\end{corollary}

Theorem~3.1 in~\cite{Vernikov-07a} and Theorem~\ref{modular commut} show that the following is true.

\begin{corollary}
\label{modular commut in SEM and EPI}
A periodic commutative semigroup variety is a modular element of the lattice $\mathbf{SEM}$ if and only if $\mathcal V$ is a modular element of the lattice $\mathbf{EPI}$.\qed
\end{corollary}

\section{Lower-modular varieties}
\label{lower-modular proof}

Here we prove Theorem~\ref{lower-modular}. To achieve this aim, we need the following

\begin{proposition}
\label{lower-modular is periodic}
If an epigroup variety $\mathcal V$ is a lower-modular element of the lattice $\mathbf{EPI}$ then $\mathcal V$ is periodic.
\end{proposition}

\begin{proof}
If $S$ is an epigroup and $x\in S$ then $\overline{\overline x}\,=\,\overline{\overline x}\,x^\omega=\,\overline{\overline x}\cdot\overline x\,x$. Thus, every epigroup satisfies the identity
\begin{equation}
\label{x**=x**x*x}
\overline{\overline x}\,=\,\overline{\overline x}\cdot\overline x\,x\ldotp
\end{equation}

It is known (see~\cite{Shevrin-94} or~\cite{Shevrin-05}, for instance) that there is a natural number $n$ such that the $n$th power of any element $x$ in an arbitrary member $S$ of $\mathcal V$ is a group element. Therefore, $\mathcal V$ satisfies the identity $x^n=\,\overline{\overline{x^n}}$. Let $n$ be the least number with such a property. Put
$$\mathcal Y =\var\{x^ny^3x=\,\overline{\overline{x^n}}\,y^3x\}\quad\text{and}\quad\mathcal Z=\var\{x^ny^2x=x^ny^3x\}\ldotp$$
Clearly, $\mathcal{V\subseteq Y}$. Since $\mathcal V$ is lower-modular, the equality
$$\mathcal{(Z\vee V)\wedge Y=(Z\wedge Y)\vee V}$$
holds. Note that
\begin{align*}
x^ny^2x&=x^ny^3x&&\text{in the variety }\mathcal Z\\
&=\,\overline{\overline{x^n}}\,y^3x&&\text{in the variety }\mathcal Y\\
&=\,\overline{\overline{x^n}}\cdot\overline{x^n}\,x^ny^3x&&\text{in the variety }\mathcal{Z\wedge Y}\text{ by \eqref{x**=x**x*x}}\\
&=\,\overline{\overline{x^n}}\cdot\overline{x^n}\,x^ny^2x&&\text{in the variety }\mathcal Z\\
&=\,\overline{\overline{x^n}}\,y^2x&&\text{in the variety }\mathcal{Z\wedge Y}\text{ by \eqref{x**=x**x*x}}\ldotp
\end{align*}
We see that $\mathcal{Z\wedge Y}$ satisfies the identity $x^ny^2x=\,\overline{\overline{x^n}}\,y^2x$. Therefore, the identity $x^ny^2x=\,\overline{\overline{x^n}}\,y^2x$ holds in $\mathcal{(Z\wedge Y)\vee V}$ and moreover, in $\mathcal{(Z\vee V)\wedge Y}$. This means that there is a deduction of this identity from identities of the varieties $\mathcal{Z\vee V}$ and $\mathcal Y$. In particular, one of the varieties $\mathcal{Z\vee V}$ or $\mathcal Y$ satisfies a non-trivial identity of the form $x^ny^2x=w$ for some word $w$. It is evident that any identity of such a kind is false in $\mathcal Y$. Whence, it holds in $\mathcal{Z\vee V}$. It is obvious that the words $x^ny^2x$ and $x^ny^3x$ do not contain images of each other relatively to endomorphisms of the unary semigroup $F$. Hence the identity $x^ny^2x=x^ny^3x$ does not imply any identity of the form $x^ny^2x=w$ where $w$ differs from the words $x^ny^2x$ and $x^ny^3x$. In particular, all identities of this form are false in $\mathcal Z$. Therefore, the variety $\mathcal{Z\vee V}$ satisfies the identity $x^ny^2x=x^ny^3x$. In particular, this identity holds in $\mathcal V$. Thus, $\mathcal V$ satisfies the identity $x^{n+3}=x^{n+4}$, whence it is periodic. 
\end{proof}

Repeating literally arguments from the proof of Lemma~\ref{identities for modular} but referring to Proposition~\ref{lower-modular is periodic} rather than Proposition~\ref{modular is periodic}, we obtain the following

\begin{lemma}
\label{identities for lower-modular}
Let $\mathcal V$ be an epigroup variety that is a lower-modular element of the lattice $\mathbf{EPI}$ and let $u$, $v$, $s$ and $t$ be pairwise non-equivalent words of the same length depending on the same letters. If the variety $\mathcal V$ satisfies the identities $u=v$ and $s=t$ then it satisfies also the identity $u=s$.\qed
\end{lemma}

\begin{proof}[Proof of Theorem~\emph{\ref{lower-modular}}. Sufficiency]
Proposition~\ref{modular and lower-modular suf in USEM} immediately implies that a 0-reduced epigroup variety is lower-modular. It remains to refer to Corollary~\ref{join with SL or ZM or SL+ZM}.

\smallskip 

\emph{Necessity.} In view of Proposition~\ref{lower-modular is periodic}, the variety $\mathcal V$ is periodic. Now we may complete the proof repeating literally arguments from the proof of Proposition~3.3 in~\cite{Shaprynskii-12a} but referring to Proposition~\ref{decomposition} rather than Lemma~2.6 in~\cite{Shaprynskii-12a}, and to Lemma~\ref{identities for lower-modular} rather than Lemma~3.1 in~\cite{Shaprynskii-12a}.
\end{proof}

Comparing Theorems~\ref{modular suf} and~\ref{lower-modular}, we obtain the following

\begin{corollary}
\label{lower-modular implies modular}
If an epigroup variety is a lower-modular element of the lattice $\mathbf{EPI}$ then it is a modular element of this lattice.\qed
\end{corollary}

By the way, we note that neither of the five other possible interrelations between properties of being a modular variety, a lower-modular variety or an upper-modular variety holds. For instance:
\begin{itemize}
\item the variety $\var\{x^2=0,\,xy=yx\}$ is modular by Theorem~\ref{modular commut} but not lower-modular by Theorem~\ref{lower-modular};
\item the variety $\var\{xyz=0\}$ is modular and lower-modular by Theorems~\ref{modular suf} and~\ref{lower-modular} respectively but not upper-modular by Theorem~\ref{upper-modular};
\item an arbitrary Abelian periodic group variety is upper-modular by Theorem~\ref{upper-modular} but neither modular nor lower-modular by Theorems~\ref{modular nec} and~\ref{lower-modular} respectively.
\end{itemize}

Theorem~1.1 in~\cite{Shaprynskii-Vernikov-10} and Theorem~\ref{lower-modular} show that the following is true.

\begin{corollary}
\label{lower-modular in SEM and EPI}
A periodic semigroup variety $\mathcal V$ is a lower-modular element of the lattice $\mathbf{SEM}$ if and only if $\mathcal V$ is a lower-modular element of the lattice $\mathbf{EPI}$.\qed
\end{corollary}

\section{An application to definable varieties}
\label{definafle}

Here we discuss an interesting application of Theorems~\ref{neutral and costandard} and~\ref{lower-modular}. A subset $A$ of a lattice $\langle L;\vee,\wedge\rangle$ is called \emph{definable in} $L$ if there exists a first-order formula $\Phi(x)$ with one free variable $x$ in the language of lattice operations $\vee$ and $\wedge$ which \emph{defines $A$ in} $L$. This means that, for an element $a\in L$, the sentence $\Phi(a)$ is true if and only if $a\in A$. A set $X$ of semigroup [epigroup] varieties is said to be \emph{definable} if it is definable in \textbf{SEM} [respectively in \textbf{EPI}]. In this situation we will say that the corresponding first-order formula \emph{defines} the set $X$.

A number of deep results about definable varieties and sets of varieties of semigroups have been obtained in~\cite{Jezek-McKenzie-93} by Je\v{z}ek and McKenzie. In particular, it was proved there that the set of all 0-reduced varieties is definable. But the article \cite{Jezek-McKenzie-93} contains no explicit first-order formula that define this set of varieties. The simple first-order formula that define the class of all 0-reduced semigroup varieties was found in~\cite[Theorem~3.3]{Vernikov-12} or~\cite[Subsection~3.5]{Vernikov-15}. Here we are going to verify that the same formula defines the set of 0-reduced varieties in the lattice \textbf{EPI}.

Theorem~\ref{lower-modular} shows that an epigroup variety is 0-reduced if and only if it is lower-modular and does not contain the variety $\mathcal{SL}$. Obviously, the set of all lower-modular varieties is definable. It remains to define the variety $\mathcal{SL}$. Evidently, the lattices \textbf{SEM} and \textbf{EPI} have the same set of atoms. Theorem~\ref{neutral and costandard} together with the well-known description of atoms of the lattice \textbf{SEM} (see~\cite[Section~1]{Shevrin-Vernikov-Volkov-09}, for instance) imply that the lattice \textbf{EPI} contains exactly two neutral atoms, namely the varieties $\mathcal{SL}$ and $\mathcal{ZM}$. Recall that a semigroup variety $\mathcal V$ is called a \emph{chain} if the lattice $L(\mathcal V)$ is a chain. Clearly, every chain variety is periodic, whence it may be considered as epigroup variety. It is well known that the variety $\mathcal{ZM}$ is properly contained in some chain variety, while the variety $\mathcal{SL}$ is not~\cite{Sukhanov-82}. Combining the mentioned observations, we see that the class of all 0-reduced varieties may be defined as the class \textbf K of epigroup varieties with the following properties:
\begin{itemize}
\item[(i)]every member of \textbf K is a lower-modular variety;
\item[(ii)]if $\mathcal V\in\mathbf K$ and $\mathcal V$ contains some neutral atom $\mathcal A$ then $\mathcal A$ is properly contained in some chain variety.
\end{itemize}
It is evident that properties~(i) and~(ii) may be written by simple first-order formulas with one free variable. We prove the following

\begin{proposition}
\label{0-reduced is definable}
The class of all $0$-reduced epigroup varieties is definable in the lattice $\mathbf{EPI}$.\qed
\end{proposition}

\section{Neutral and costandard varieties}
\label{neutral and costandard proof}

Here we prove Theorem~\ref{neutral and costandard}. The proof will be given by the following scheme:

\smallskip

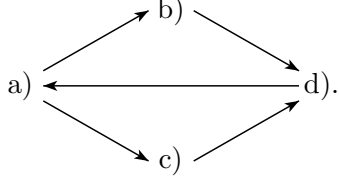
\begin{center}
\unitlength=1mm
\begin{picture}(40,20)
\gasset{AHnb=1,AHLength=2,linewidth=.2}
\drawline(3,12)(17,20)
\drawline(3,8)(17,0)
\drawline(23,20)(37,12)
\drawline(23,0)(37,8)
\drawline(37,10)(3,10)
\put(0,10){\makebox(0,0)[cc]{a)}}
\put(20,20){\makebox(0,0)[cc]{b)}}
\put(20,0){\makebox(0,0)[cc]{c)}}
\put(40,10){\makebox(0,0)[cc]{d).}}
\end{picture}
\end{center}

\smallskip

\noindent The implications a)\,$\longrightarrow$\,b) and a)\,$\longrightarrow$\,c) are evident, while the implication d)\,$\longrightarrow$\,a) immediately follows from Propositions~\ref{SL is neutral} and~\ref{ZM is neutral}, and the well-known fact that the set of all neutral elements of a lattice $L$ forms a sublattice in $L$ (see~\cite[Theorem~259]{Gratzer-11}). It remains to verify the implications b)\,$\longrightarrow$\,d) and c)\,$\longrightarrow$\,d).

Suppose that $\mathcal V$ is either costandard or simultaneously lower-modular and upper-modular. We need to check that $\mathcal V$ coincides with one of the varieties $\mathcal T$, $\mathcal{SL}$, $\mathcal{ZM}$ or $\mathcal{SL\vee ZM}$. In any case, $\mathcal V$ is modular (this is evident whenever $\mathcal V$ is costandard, and follows from Corollary~\ref{lower-modular implies modular} whenever $\mathcal V$ is lower-modular and upper-modular). Then we may apply Theorem~\ref{modular nec} and conclude that $\mathcal{V=M\vee N}$ where $\mathcal M$ is one of the varieties $\mathcal T$ or $\mathcal{SL}$ and $\mathcal N$ is a nilvariety. It remains to verify that $\mathcal N$ is one of the varieties $\mathcal T$ or $\mathcal{ZM}$. By Lemma~\ref{SL and ZM are atoms} we have to check that $\mathcal{N\subseteq ZM}$. In other words, we need to verify that $\mathcal N$ is a variety of degree $\le2$. In view of Corollary~\ref{join with SL or ZM or SL+ZM}, $\mathcal N$ is either costandard or simultaneously lower-modular and upper-modular.

Suppose at first that $\mathcal N$ is costandard. Then it is codistributive. Now Lemma~\ref{codistr without P and P*} applies with the desirable conclusion. We have proved the implication b)\,$\longrightarrow$\,d).

It remains to consider the case when $\mathcal N$ is lower-modular and upper-modular. The variety $\mathcal N$ is periodic, whence it may be considered as a semigroup variety. Therefore, $\mathcal N$ is a lower-modular and upper-modular element of the lattice \textbf{PER}.

It follows from~\cite[Theorem~1]{Vernikov-07b} that if a nilvariety of semigroups is lower-modular in \textbf{SEM} then it is 0-reduced. All varieties that appear in the proof of this fact are periodic. So, a nilvariety is 0-reduced whenever it is a lower-modular element in \textbf{PER}. In particular, the variety $\mathcal N$ is 0-reduced.

Suppose that $\mathcal N$ is a variety of degree $>2$. It follows from~\cite[Theorem~1]{Vernikov-08c} that if a proper semigroup variety of degree $>2$ is upper-modular in \textbf{SEM} then it is commutative. All varieties that appear in the proof of this fact are periodic again. So, a variety of degree $>2$ is commutative whenever it is an upper-modular element in \textbf{PER}. In particular, the variety $\mathcal N$ is commutative.

Thus, $\mathcal N$ is a 0-reduced and commutative nilvariety. Therefore, $\mathcal{N\subseteq ZM}$. This completes the proof of the implication c)\,$\longrightarrow$\,d), and hence the proof of Theorem~\ref{neutral and costandard} as a whole.\qed

\medskip

Theorem~\ref{neutral and costandard} immediately implies the following

\begin{corollary}
\label{costandard implies standard}
If an epigroup variety is a costandard element of the lattice $\mathbf{EPI}$ then it is a standard element of this lattice.\qed
\end{corollary}

Comparing Theorem~1.3 in~\cite{Vernikov-11} with Theorem~\ref{neutral and costandard}, we obtain the following

\begin{corollary}
\label{neutral and costandard in SEM and EPI}
For a periodic semigroup variety $\mathcal V$, the following are equivalent:
\begin{itemize}
\item[\textup{a)}]$\mathcal V$ is a costandard element of the lattice $\mathbf{EPI}$;
\item[\textup{b)}]$\mathcal V$ is a neutral element of the lattice $\mathbf{EPI}$;
\item[\textup{c)}]$\mathcal V$ is a costandard element of the lattice $\mathbf{SEM}$;
\item[\textup{d)}]$\mathcal V$ is a neutral element of the lattice $\mathbf{SEM}$.\qed
\end{itemize}
\end{corollary}

\section{Distributive and standard varieties}
\label{distributive and standard proof}

Here we prove Theorem~\ref{distributive and standard}. The implication b)\,$\longrightarrow$\,a) of this theorem is evident, whence it suffices to verify the implications a)\,$\longrightarrow$\,c), c)\,$\longrightarrow$\,a) and a)\,$\longrightarrow$\,b).

\smallskip

a)\,$\longrightarrow$\,b) Suppose that an epigroup variety $\mathcal V$ is distributive. Hence it is lower-modular. Corollary~\ref{lower-modular implies modular} implies that $\mathcal V$ is modular. Now Lemma~\ref{distr+mod=stand} implies with the conclusion that $\mathcal V$ is standard.

\smallskip

a)\,$\longrightarrow$\,c) We denote by $F_m$ the free unary semigroup over the alphabet $\{x_1,x_2,\dots,x_m\}$ and by $S_m$ the symmetric group on the set $\{1,2,\dots,m\}$. If $\sigma\in S_m$ and $u\in F_m$ then $\sigma(u)$ denotes the image of the word $u$ under the automorphism of the unary semigroup $F_m$ induced by the action of the permutation $\sigma$ on indexes of letters.The following statement easily follows from the proof of Lemma~5.1 of the paper~\cite{Vernikov-Shaprynskii-10}.

\begin{lemma}
\label{two semigroup words}
Let $u$ and $v$ be semigroup words from $F_m$ such that the identity $u=v$ is non-balanced and none of the words $u$ and $v$ contains the image of other one over some endomorphism of the free unary semigroup, and let $\sigma$ be a permutation from $S_m$. If the variety $\mathcal X=\var\{u=v\}$ satisfies a non-trivial identity of the form $\sigma(u)=w$ where $w$ is a semigroup word then $w\equiv\sigma(v)$. \qed
\end{lemma}

Note that if the identity $u=v$ is balanced then the class $K_{\{u=v\}}$ is not a variety by Lemma~\ref{epigroup variety}. Thus, the restriction to the identity $u=v$ to be non-balanced in Lemma~\ref{two semigroup words} is natural.

The following assertion is an `epigroup analog' of Lemma~5.2 of the article~\cite{Vernikov-Shaprynskii-10}.

\begin{lemma}
\label{if u=0 then v=0}
Let $\mathcal V$ be an epigroup variety that is a distributive element of the lattice $\mathbf{EPI}$ and let $u$, $v$ be semigroup words such that $c(u)=c(v)$ and one of the following holds:
\begin{itemize}
\item[\textup{(i)}] none of the words $u$ and $v$ contains the image of other one over some endomorphism of the free unary semigroup;
\item[\textup{(ii)}] $\ell(u)=\ell(v)$.
\end{itemize}
If the variety $\mathcal V$ satisfies the identity $u=0$ then it satisfies the identity $v=0$.
\end{lemma}

\begin{proof}
We may assume that $u,v\in F_m$ for some natural $m$. Suppose that $\mathcal V$ satisfies the identity $u=0$.

\smallskip

(i) If $m=1$ then $u\equiv x^q$ and $v\equiv x^r$ for some $q$ and $r$. Then one of the words $u$ and $v$ contains an image of another one that contradicts the hypothesis. Therefore, $m >1$, whence the group $S_m$ contains a non-trivial permutation $\alpha$. Put $w\equiv\alpha(u)$. Then $\mathcal V$ satisfies the identity $w=0$ that implies $u=w$. If the identity $u=v$ is non-balanced then we may complete the proof repeating literally arguments from the corresponding part of the proof of item~(i) in Lemma~5.2 of the paper~\cite{Vernikov-Shaprynskii-10}. Suppose now that the identity $u=v$ is balanced. Let $x$ be a letter with $x\notin c(u)$. Then the identities $xu=v$ and $v=xw$ are non-balanced. According to Lemma~\ref{epigroup variety}, we may consider the varieties $\mathcal Y=\var\{xu=v\}$ and $\mathcal Z=\var\{v=xw\}$. Since the variety $\mathcal V$ is distributive, $\mathcal {V \vee (Y \wedge Z)=(V \vee Y) \wedge (V \vee Z)}$. Clearly, the variety $\mathcal{V\vee(Y\wedge Z)}$ satisfies the identity $xu=xw$. Therefore, this identity holds in the variety $\mathcal{(V\vee Y)\wedge(V\vee Z)}$. Then there exists a deduction of the identity $xu=xw$ from identities of the varieties $\mathcal{V\vee Y}$ and $\mathcal{V \vee Z}$. In particular, one of these varieties satisfies a non-trivial identity of the form $xu=w_1$.

Suppose at first that $xu=w_1$ holds in $\mathcal{V\vee Y}$. Then it holds in $\mathcal Y$. Now Lemma~\ref{two semigroup words} with the trivial permutation $\sigma$ from $S_m$ applies, and we conclude that $w_1\equiv v$. Since $xu=w_1$ in $\mathcal V$, we have that $\mathcal V$ satisfies the identity $v=xu$ that implies $v=0$. It remains to consider the case when the identity $xu=w_1$ holds in $\mathcal{V\vee Z}$. In particular, it holds in $\mathcal Z$. Let $\sigma$ be a permutation of the alphabet such that the restriction of $\sigma$ on the set $c(u)$ coincides with $\alpha^{-1}$ and $\sigma(x)\equiv x$. Now we apply Lemma~\ref{two semigroup words} with such permutation $\sigma$. Then we obtain that $w_1\equiv\sigma(v)$. But the identity $xu=w_1$ holds in $\mathcal V$. Hence $\mathcal V$ satisfies the identity $xu=\sigma(v)$ that implies $\sigma(v)=0$ and $v=0$.

\smallskip

(ii) Here we may repeat literally arguments from the proof of item (ii) in~\cite[כוללא~5.2]{Vernikov-Shaprynskii-10}.
\end{proof}

Now we are well prepared to complete the proof of the implication a)\,$\longrightarrow$\,c) of Theorem~\ref{distributive and standard}. Let $\mathcal V$ be a distributive epigroup variety. Then $\mathcal V$ is a lower-modular variety. In view of Theorem~\ref{lower-modular}, $\mathcal{V=M\vee N}$ where $\mathcal M$ is one of the varieties $\mathcal T$ or $\mathcal{SL}$ and $\mathcal N$ is a 0-reduced variety. Corollary~\ref{join with SL or ZM or SL+ZM} implies that the variety $\mathcal N$ is distributive. Being 0-reduced, this variety satisfies the identity $u=0$ for some word $u$. We may assume that this identity fails in the class of all nil-semi\-groups. Lemma~\ref{identity for nil} implies that $u$ is a semigroup word. We may assume that $c(u)=\{x,y\}$. Indeed, if $u$ depends on one letter then we may substitute $xy$ to this letter, and if $u$ depends on two or more letters then we may substitute $x$ to one of these letters and $y$ to all other of them. Now we may repeat literally the corresponding part of the proof of Proposition~3.2 of the article~\cite{Vernikov-Shaprynskii-10} with referring to Lemmas~\ref{two semigroup words} and~\ref{if u=0 then v=0} of the present article rather than Lemmas~5.1 and~5.2 of the paper~\cite{Vernikov-Shaprynskii-10} respectively. As a result, we obtain that $\mathcal N$ satisfies all identities of the form $v=0$ with $c(v)=c(u)$ and $\ell(v)\ge3$. In particular, $\mathcal N$ satisfies the identities $x^2y=xyx=yx^2=0$. We see that $\mathcal N$ is a 0-reduced subvariety of the variety $\mathcal Q$. If $\mathcal{N=Q}$ then we are done. Let now $\mathcal{N\subset Q}$. Then $\mathcal N$ is given in $\mathcal Q$ by some set of 0-reduced identities. By Lemma~\ref{what =0 in Q}, $\mathcal N$ is given within $\mathcal Q$ either by the identity $x^2=0$ or by the identity $x_1x_2\cdots x_n=0$ for some $n$ or by these two identities together. Thus $\mathcal N$ is one of the varieties $\mathcal Q$, $\mathcal Q_n$, $\mathcal R$ or $\mathcal R_n$. The implication a)\,$\longrightarrow$\,c) is proved.

\smallskip

c)\,$\longrightarrow$\,a) Let $\mathcal{V=M\vee N}$ where $\mathcal M$ is one of the varieties $\mathcal T$ or $\mathcal {SL}$ and $\mathcal N$ is one of the varieties $\mathcal Q$, $\mathcal Q_n$, $\mathcal R$ or $\mathcal R_n$. In view of Corollary~\ref{join with SL or ZM or SL+ZM}, it suffices to prove that the variety $\mathcal N$ is standard. In other words, we may assume that $\mathcal V$ is one of the varieties $\mathcal Q$, $\mathcal Q_n$, $\mathcal R$ or $\mathcal R_n$. In particular, $\mathcal V$ is periodic, whence it may be considered as a semigroup variety.

Let $\mathcal Y$ and $\mathcal Z$ be arbitrary epigroup varieties. We need to verify that
$$\mathcal{V\vee(Y\wedge Z)=(V\vee Y)\wedge(V\vee Z)}\ldotp$$
It suffices to check that $\mathcal{(V\vee Y)\wedge(V\vee Z)\subseteq V\vee(Y\wedge Z)}$ because the contrary inclusion is evident. In other words, we have to prove that an arbitrary identity holds in $\mathcal{(V\vee Y)\wedge(V\vee Z)}$ whenever it holds in $\mathcal{V\vee(Y\wedge Z)}$. In view of~\cite[Lemma~2.2]{Vernikov-Shaprynskii-10} we may suppose that $\mathcal{Y,Z\supseteq SL}$.

Let $u=v$ be an identity that holds in $\mathcal{V\vee(Y\wedge Z)}$. Then it holds in $\mathcal V$ and there exists a deduction of this identity from identities of the varieties $\mathcal Y$ and $\mathcal Z$. Let
\begin{equation}
\label{deduction}
u \equiv w_0 \longrightarrow w_1 \longrightarrow \cdots \longrightarrow w_k \equiv v
\end{equation}
be the shortest such deduction. Since $\mathcal{Y,Z\supseteq\mathcal SL}$, Lemma~\ref{word problem}(i) implies that $c(w_0)=c(w_1)=\cdots=c(w_k)$.

We use an induction by $k$. The induction base is evident because if $k=1$ then the identity $u=v$ holds in one of the varieties $\mathcal{V\vee Y}$ or $\mathcal{V\vee Z}$, whence it holds in the intersection of these varieties. Let now $k>1$. The identity $u=v$ holds in the variety $\mathcal V$. Since $\mathcal V$ is 0-reduced, $u=v=0$ holds in $\mathcal V$. If $w_i=0$ in $\mathcal V$ for some $0<i<k$ then our considerations may be completed by repeating literally arguments from the corresponding part of the proof of Proposition~3.3 of the paper~\cite{Vernikov-Shaprynskii-10}. In view of Lemma~\ref{what =0 in Q}, we may assume that each of the words $w_1,w_2,\dots,w_{k-1}$ either is a linear word or coincides with the word $x^2$.

Suppose that $w_i\equiv x^2$ for some $0<i<k$. The choice of the deduction~\eqref{deduction} guarantees that the words $w_0,w_1,\dots,w_k$ are pairwise different. Since $c(w_0)=c(w_1)=\cdots=c(w_k)=\{x\}$, we have that each of the varieties $\mathcal Y$ and $\mathcal Z$ satisfies a non-trivial identity of the form $x^m=x^n$. Therefore, the varieties $\mathcal Y$ and $\mathcal Z$ are periodic. Hence they may be considered as semigroup varieties. This permits to complete our considerations by repeating literally arguments from the corresponding part of the proof of Proposition~3.3 of the paper~\cite{Vernikov-Shaprynskii-10}.

It remains to consider the case when the words $w_1,w_2,\dots,w_{k-1}$ are linear. We may assume that the words $u$ and $v$ are non-linear (as usual, this claim may be verified by the same arguments as in the proof of~\cite[Proposition~3.3]{Vernikov-Shaprynskii-10}). Besides that, we will assume without loss of generality that the identity $u=w_1$ holds in $\mathcal Y$ and therefore, $w_1=w_2$ holds in $\mathcal Z$. Further considerations are divided into three cases.

\smallskip

\emph{Case}~1: $k=2$. Here the deduction~\eqref{deduction} has the form $u\rightarrow w_1\rightarrow v$, $u=w_1$ in $\mathcal Y$ and $w_1=v$ in $\mathcal Z$. Since the words $u$ and $v$ are non-linear, the varieties $\mathcal Y$ and $\mathcal Z$ are periodic. Hence they may be considered as semigroup varieties. This permits to complete a consideration of this case by repeating literally arguments from the Case~1 in the proof of Proposition~3.3 of the paper~\cite{Vernikov-Shaprynskii-10}.

\smallskip

\emph{Case}~2: $k=3$. Here the deduction~\eqref{deduction} has the form $u\rightarrow w_1\rightarrow w_2\rightarrow v$ and the identities $u=w_1$ and $w_2=v$ hold in $\mathcal Y$. Since the words $u$ and $v$ are non-linear, the variety $\mathcal Y$ is periodic. Hence it may be considered as a semigroup variety. This variety satisfies the identity $x_1x_2\cdots x_m=u$ and $\ell(u)>m$. Therefore, $\mathcal Y$ is a variety of degree $\le m$ (see~\cite[Lemma~1]{Sapir-Sukhanov-81}). According to~\cite[Proposition~2.11]{Vernikov-08b}, $\mathcal Y$ satisfies an identity of the form
$$x_1x_2\cdots x_m=x_1x_2\cdots x_{i-1}( x_i \cdots x_j)^tx_{j+1}\cdots x_m$$
for some $t>1$ and $0\le i\le j\le m$. This permits to complete a consideration of this case by repeating literally arguments from the Case~2 in the proof of Proposition~3.3 of the paper~\cite{Vernikov-Shaprynskii-10}.

\smallskip

\emph{Case}~3: $k>3$. This case may be considered by repeating literally arguments from the Case~3 in the proof of Proposition~3.3 of the paper~\cite{Vernikov-Shaprynskii-10}.

\smallskip

We complete the proof of the implication c)\,$\longrightarrow$\,a) and of Theorem~\ref{distributive and standard} as a whole.\qed

\medskip

Comparing Theorem~1.1 in~\cite{Vernikov-Shaprynskii-10} with Theorem~\ref{distributive and standard}, we obtain the following

\begin{corollary}
\label{distributive and standard in SEM and EPI}
For a periodic semigroup variety $\mathcal V$, the following are equivalent:
\begin{itemize}
\item[\textup{a)}]$\mathcal V$ is a distributive element of the lattice $\mathbf{EPI}$;
\item[\textup{b)}]$\mathcal V$ is a standard element of the lattice $\mathbf{EPI}$;
\item[\textup{c)}]$\mathcal V$ is a distributive element of the lattice $\mathbf{SEM}$;
\item[\textup{d)}]$\mathcal V$ is a standard element of the lattice $\mathbf{SEM}$.\qed
\end{itemize}
\end{corollary}

\section{Open questions}
\label{questions}

It is verified in~\cite[Theorem~1.1]{Vernikov-11} that if a proper semigroup variety $\mathcal V$ is codistributive in \textbf{SEM} then the square of any member of $\mathcal V$ is completely regular. We do not know, whether the epigroup analog of this fact is true.

\begin{question}
\label{codistr nec?}
Is it true that if an epigroup variety $\mathcal V$ is a codistributive element of the lattice $\mathbf{EPI}$ then the square of any member of $\mathcal V$ is completely regular\textup? 
\end{question}

As we have already mentioned in Section~\ref{upper-modular proof}, it is proved in~\cite[Theorem~1.1]{Vernikov-08b} that if $\mathcal V$ is a proper upper-modular in \textbf{SEM} semigroup variety then every nilsubvariety of $\mathcal V$ is commutative and satisfies the identity~\eqref{xxy=xyy}. Proposition~\ref{umod nec} gives a partial epigroup analog of this assertion. We do not know, whether the full analog is true.

\begin{question}
\label{umod nil is commut?}
Suppose that an epigroup variety $\mathcal V$ is an upper-modular element of the lattice $\mathbf{EPI}$ and let $\mathcal N$ be a nilsubvariety of $\mathcal V$. Is it true that the variety $\mathcal N$
\begin{itemize}
\item[\textup{a)}] is commutative;
\item[\textup{b)}] satisfies the identity~\eqref{xxy=xyy}\textup?
\end{itemize} 
\end{question}

Proposition~\ref{umod nec} shows that the affirmative answer to Question~\ref{umod nil is commut?}a) would immediately implies the same answer to Question~\ref{umod nil is commut?}b).

Further, it is verified in~\cite[Theorem~1]{Vernikov-08c} that every proper upper-modular in \textbf{SEM} variety is either commutative or has a degree $\le2$. We do not know, whether the epigroup analog of this alternative is valid. We formulate the corresponding question together with its weaker version.

\begin{question}
\label{umod alternative?}
Suppose that an epigroup variety $\mathcal V$ is an upper-modular element of the lattice $\mathbf{EPI}$. Is it true that the variety $\mathcal V$
\begin{itemize}
\item[\textup{a)}] either is commutative or has a degree $\le2$;
\item[\textup{b)}] either is permutative or has a finite degree\textup?
\end{itemize}
\end{question}

The affirmative answer to Question~\ref{umod alternative?}a) together with Theorem~\ref{upper-modular} would immediately imply a complete description of upper-modular epigroup varieties of degree $>2$.

If $\mathcal V$ is a proper upper-modular in \textbf{SEM} variety then the lattice $L(\mathcal V)$ is modular~\cite[Corolary~2]{Vernikov-08c} and every subvariety of $\mathcal V$ is upper-modular in \textbf{SEM}~\cite[Corolary~3]{Vernikov-08c}. This makes a natural the following

\begin{question}
\label{umod mod&hered?}
Suppose that an epigroup variety $\mathcal V$ is an upper-modular element of the lattice $\mathbf{EPI}$. Is it true that
\begin{itemize}
\item[\textup{a)}] the lattice $L(\mathcal V)$ is modular;
\item[\textup{b)}] every subvariety of $\mathcal V$ is an upper-modular element of the lattice $\mathbf{EPI}$\textup?
\end{itemize}
\end{question}

We note that the affirmative answer to Question~\ref{umod mod&hered?}a) would immediately implies the same answer to Question~\ref{umod mod&hered?}b) (see~\cite[Lemma~2.1]{Vernikov-08b}).

\end{document}